\documentclass[a4paper,12pt]{article}
\usepackage{amsmath,amsfonts,amssymb,amsthm,amscd}
\usepackage[arrow, matrix, curve]{xy}
\usepackage[colorlinks=true,citecolor=blue]{hyperref}

\font\cmssl=cmss10 at 12 pt

\usepackage{slashed}
\usepackage[normalem]{ulem}

\newtheorem{thm}{Theorem}
\newtheorem{lem}[thm]{Lemma}
\newtheorem{prop}[thm]{Proposition}
\newtheorem{defn}[thm]{Definition}
\newtheorem{cor}[thm]{Corollary}
\newtheorem{rem}[thm]{Remark}
\newtheorem{notation}[thm]{Notation}
\newtheorem{exa}[thm]{Example}

\title{$B_{n}$-generalized pseudo-K\"{a}hler structures}

\date{\today}

\author{Vicente Cort\'es and  Liana David}

\begin{document}

\maketitle

{\bf Abstract:}  We define the  notions  of 
$B_{n}$-generalized pseudo-Hermitian and 
 $B_{n}$-generalized pseudo-K\"{a}hler structure
on  an odd exact Courant algebroid $E$. 
When $E$ is in the standard form (or of type $B_{n}$) 
we express these notions in terms of classical tensor fields on the base of $E$. 
This is analogous to the  bi-Hermitian viewpoint on generalized K\"{a}hler structures  on
exact Courant algebroids.
We describe left-invariant $B_{n}$-generalized pseudo-K\"{a}hler structures on Courant algebroids of type $B_{n}$ over Lie groups of dimension
two, three and four.

\tableofcontents

\section{Introduction}

Generalized complex geometry is  a unification of complex and symplectic geometry and represents 
an active research area in current mathematics at the interface with mathematical physics. The  initial idea was  to replace the tangent bundle of a manifold $M$ with the generalized tangent bundle $\mathbb{T}M:= TM \oplus T^{*}M$ and to include 
complex and symplectic structures in a more general type of structure,  a so called generalized complex structure. 
Later on, other
classical structures  (like K\"{a}hler, quaternionic, hyper-K\"{a}hler  etc.) were  defined  and studied in this general setting.

Generalized tangent bundles are the simplest class of Courant algebroids. They are sometimes called exact Courant algebroids, 
owing to
the fact that  they  fit into an exact sequence 
$$
0\rightarrow  T^{*}M \stackrel{ \pi^{*}}{\rightarrow }\mathbb{T}M \stackrel{\pi}{\rightarrow}  TM \rightarrow 0
$$
where $\pi : \mathbb{T}M \rightarrow TM$  is 
the natural projection 
(the anchor of $\mathbb{T}M$)
and $\pi^{*} : T^{*}M \rightarrow  \mathbb{T}M$ is the dual of $\pi$
(here $\mathbb{T}M$ and
$(\mathbb{T}M)^{*}$ are identified using the  natural scalar product  of neutral signature of $\mathbb{T}M$). 
The notion of a Courant
algebroid 
(see Section \ref{preliminary} for its definition) 
was defined  for the first time   by Z.\ J.\ Liu, A.\ Weinstein and P.\ Xu (see  \cite{w}) and since then it has been  intensively studied.  
Regular Courant algebroids  (i.e.\ Courant  algebroids for which the  image of the anchor is a vector bundle) were classified in \cite{chen}.  This includes a classification of transitive Courant algebroids (i.e.\ Courant algebroids with surjective anchor). 
In the context of $T$-duality, 
heterotic Courant
algebroids (i.e.\ Courant algebroids for which the  associated quadratic Lie algebra bundle is an adjoint bundle) were considered in \cite{baraglia}.
To any regular Courant algebroid $E$ one can associate a  (regular) quadratic Lie algebroid $E/ ( \mathrm{Ker}\,  \pi)^{\perp}$, where $\perp$ denotes the orthogonal  complement with respect to the scalar product of $E$.  Conversely, a regular quadratic Lie algebroid  arises in this
way from a Courant algebroid if and only if its first Pontryagin class \cite{bressl, severa} vanishes  (see Theorem 1.10 of  \cite{chen}).

Our setting in this paper are the  so called 
Courant algebroids of type $B_{n}$
(and their global analogue, the odd exact Courant algebroids),  introduced in \cite{rubio}. 
The terminology is justified by the fact that the underlying bundle of a Courant algebroid of type $B_{n}$ is of the form
$TM\oplus T^{*}M \oplus \mathbb{R}$ 
and its scalar product is  the natural scalar product of $TM\oplus T^{*}M \oplus \mathbb{R}$ 
of signature $(n+1, 1)$, where  $n$ is the dimension of the manifold $M$. 
(Along the same lines, exact Courant algebroids were called in \cite{rubio} Courant algebroids of type $D_{n}$, as the  natural scalar product
of a generalized tangent bundle has signature $(n, n)$).  
One  may  also define a Courant algebroid  of type $B_{n}$ as a transitive Courant algebroid in the standard form, 
with  quadratic Lie algebra bundle $\mathcal{Q} = M\times \mathbb{R}$ 
the trivial line bundle,  
with  positive definite canonical  metric,
trivial Lie bracket
and canonical flat connection.  
Therefore, from  the  viewpoint  of the classification  of  transitive Courant algebroids \cite{chen}, the Courant algebroids of type 
$B_{n}$ represent the next simplest class, after generalized tangent bundles. \

Courant  algebroids of type $B_{n}$ may be seen as the odd analogue of  
Courant algebroids of type $D_{n}$. With this analogy in mind, 
generalized  complex structures on odd exact Courant algebroids   (the so called $B_{n}$-generalized complex structures) 
were introduced and  studied systematically in \cite{rubio}.
It is worth  mentioning that there are various approaches  to define odd dimensional analogues of $D_{n}$-geometry
(see e.g. \cite{gen-1, gen-2, gen-3}).
 In this paper we adopt   the  viewpoint of \cite{rubio} and, in analogy with \cite{gualtieri-thesis},   we define generalized pseudo-K\"{a}hler structures on  odd exact Courant algebroids
 as an enrichment of $B_{n}$-generalized complex structures.\\

{\bf Structure of the paper and main results.} In Section \ref{preliminary} we recall basic definitions 
on Courant algebroids of type $B_{n}$, odd exact Courant algebroids, 
generalized metrics and $B_{n}$-generalized complex structures.\

In Section   \ref{definitions} we define  the notion of $B_{n}$-generalized almost pseudo-Hermitian structure on  an odd exact Courant algebroid $E$ as  a pair  $(\mathcal{G}, \mathcal F)$ formed 
by a generalized metric  $\mathcal{G}$ and a $B_{n}$-generalized almost complex structure  $\mathcal F$ which satisfy 
$\mathcal{G}^{\mathrm{end}} \mathcal F = \mathcal F \mathcal{G}^{\mathrm{end}}$.
The integrability of $(\mathcal G , \mathcal F )$ (and the   corresponding notion of 
$B_{n}$-generalized  pseudo-K\"{a}hler structure)  is defined by imposing integrability   on  $\mathcal F$ and on
the second $B_{n}$-generalized almost complex structure 
$\mathcal F_{2}:=  \mathcal F \mathcal{G}^{\mathrm{end}}$.\

In analogy with the bi-Hermitian viewpoint on generalized K\"{a}hler structures  on generalized tangent bundles    \cite{gualtieri-thesis}, 
 in Section  \ref{components-hermitian}  we assume that $E$ is a Courant algebroid of type $B_{n}$ 
 and 
 we describe 
 $B_{n}$-generalized almost pseudo-Hermitian structures on $E$ 
 in terms of classical tensor fields 
(called  ``components'')
on the base of $E$.  The case when $n$ is odd is described in the next proposition (below ''$\perp$''  denotes the orthogonal complement
with respect to $g$), cf.\ Proposition~\ref{data-on-M} i).

\begin{prop}\label{data-on-M-odd} 
Let $E$ be a Courant algebroid of type  $B_{n}$ over an $n$-di\-men\-sio\-nal manifold $M$, with Dorfman bracket twisted by
$(H, F).$  Assume that $n$ is odd.  A $B_{n}$-generalized  almost  pseudo-Hermitian  structure 
$(\mathcal{G}, \mathcal F)$ on $E$
is equivalent to the data  $(g, J_{+}, J_{-}, X_{+}, X_{-})$ 
where $g$ is a pseudo-Riemannian metric on $M$, 
$J_{\pm}\in \Gamma ( \mathrm{End}\, TM)$ are $g$-skew-symmetric endomorphisms and 
$X_{\pm}\in {\mathfrak X}(M)$ are  vector fields of norm  one,   such that  $J_{\pm } X_{\pm} =0$ and 
$J_{\pm}\vert_{ X_{\pm}^{\perp}}$ are complex structures.
\end{prop}

 The case when $n$ is even is described in the next proposition, cf.\ Proposition~\ref{data-on-M} ii). 
 
 \begin{prop}\label{data-on-M-even} 
 Let $E$ be a Courant algebroid of type $B_{n}$ over an $n$-di\-men\-sio\-nal manifold $M$, with Dorfman bracket twisted by
$(H, F).$   Assume that $n$ is even. A $B_{n}$-generalized almost pseudo-Hermitian structure 
$(\mathcal{G}, \mathcal F)$ 
on $E$
is equivalent to the data $(g,  J_{+},  J_{-}, X_{+} , X_{-},  c_{+})$ where $g$ is a pseudo-Riemannian  metric on $M$,
$J_{\pm}\in \Gamma ( \mathrm{End}\, TM)$ are  $g$-skew-symmetric endomorphisms, 
$X_{\pm}\in {\mathfrak X}(M)$  are $g$-orthogonal vector fields and
$c_{+}\in C^{\infty}(M)$ is a function, such that
$J_{-}$ is an almost complex structure on $M$, 
$J_{+}$ satisfies 
\begin{align}
\nonumber&  J_{+}X_{+} = - c_{+} X_{-},\ J_{+}X_{-} = c_{+}X_{+},\\
\nonumber& J_{+}^{2} X = - X + g(X, X_{+}) X_{+} + g(X, X_{-}) X_{-},\ \forall X\in TM
\end{align}
and $g(X_{+}, X_{+}) = g(X_{-}, X_{-})= 1-c_{+}^{2}$.
\end{prop}

Our aim in  Section \ref{components-kahler} is to express the integrability of a $B_{n}$-generalized almost pseudo-Hermitian structure 
$(\mathcal G , \mathcal F)$ 
on a Courant algebroid of type $B_{n}$ in terms of its  components.  When $n$ is odd we obtain the
following  summarized description. (Below $\nabla$ denotes the Levi-Civita connection of $g$ and 
$T^{(1,0)}_{A}M\subset (TM)_{\mathbb{C}}$  denotes the $i$-eigenbundle of an endomorphism $A$ of $TM$). 

\begin{thm}\label{integr-odd-introd}  In the setting of Proposition \ref{data-on-M-odd}, define connections $\nabla^\pm$ 
on $TM$ by
\begin{align}
\nonumber& \nabla^{-}_{X} := \nabla_{X} +\frac{1}{2} H(X),\\
\nonumber& \nabla^{+}_{X} X := \nabla_{X} -\frac{1}{2} H(X) - J_{+} F(X) \otimes X_{+},
\end{align}
for any $X\in {\mathfrak X}(M)$, where we are identifying vectors and co-vectors using the metric, as explained
after equation (\ref{nabla-1-def}), cf.\ Notation~\ref{H-F-notation}.  
Then $(\mathcal G , \mathcal F)$ is a $B_{n}$-generalized pseudo-K\"{a}hler structure if and only if $\nabla^{\pm} $ preserves 
$T_{J_{\pm}}^{(1,0)} M$, 
\begin{equation}
\nabla^{-}_{X}X_{-} =0,\ \nabla^{+}_{X} X_{+} = - J_{+} F(X),\ \forall X\in {\mathfrak X}(M)
\end{equation}
and certain algebraic conditions hold (see  Theorem 
\ref{integrability-odd} for the precise conditions).
\end{thm}

The analogous result for  $n$ is even is stated as follows.

\begin{thm}\label{integr-even-introd}  In the setting of Proposition \ref{data-on-M-even}, 
assume that $X_{+}$ and $X_{-}$ are non-null (at any point), i.e. $c_{+} (p)^{2}  \neq  1$ for any $p\in M.$ 
Define connections $D^\pm$ on $TM$ by
\begin{align}
\nonumber& D^{-}_{X} := \nabla_{X} +\frac{1}{2} H(X),\\
\nonumber& D^{+}_{X} X := \nabla_{X} -\frac{1}{2} H(X) +\frac{c_{+}}{ 1- c_{+}^{2}} F(X)\otimes X_{+} -\frac{ J_{+} F(X)}{ 1- c_{+}^{2}} \otimes 
X_{-},
\end{align}
for any $X\in {\mathfrak X}(M).$  Then $(\mathcal G , \mathcal F)$ is a $B_{n}$-generalized pseudo-K\"{a}hler structure if and only if 
$D^{\pm} $ preserves 
$T_{J_{\pm}}^{(1,0)} M$, $X_{+}$ and $X_{-}$ commute, 
their  covariant derivatives  are given by
\begin{align}
\nonumber& \nabla_{X}  X_{+} =-\frac{1}{2} ( i_{X_{+}} H) (X) + c_{+} F(X),\\
\nonumber&  \nabla_{X} X_{-} = - \frac{1}{2}( i_{X_{-}} H)(X) - J_{+} F(X),
\end{align}
$i_{X_{+}} F = dc_{+}$  
and certain algebraic conditions hold
(see Theorem  \ref{1-even} for the precise conditions). 
\end{thm}

The proofs of Theorems \ref{integr-odd-introd} and  \ref{integr-even-introd}  are analogues  to the proof of
Gualtieri on the relation between generalized K\"{a}hler structures on exact Courant algebroids and bi-Hermitian data
\cite{gualtieri-thesis}.  More precisely, we find a criterion for the integrability of a $B_{n}$-generalized  almost pseudo-Hermitian structure
$(\mathcal G , \mathcal F)$ in terms of the  (Dorfman) integrability of the bundles $L_{1}$, $L_{1}\cap L_{2}$  and $L_{1}\cap \bar{L}_{2}$
and  their invariance under the Dorfman Lie derivative ${\mathbf L}_{u_{0}}$ 
(where $L_{1}$ is the $(1,0)$-eigenbundle of $\mathcal F$,  $L_{2}$  is the $(1,0)$-eigenbundle of $G^{\mathrm{end}} \mathcal F$ and $u_{0} \in \Gamma (\mathrm{Ker}\, \mathcal F )$ is  suitably normalized),  
see Proposition \ref{criterion}. The proofs of the above theorems follow by
computing  the intersections $L_{1}\cap L_{2}$ and $L_{1} \cap \bar{L}_{2}$ in terms of the components of 
$(\mathcal G , \mathcal F)$ and applying  the above criterion. 
The assumption  from Theorem \ref{integr-even-introd} that $X_{\pm}$ are non-null simplify the argument considerably. 
Along the way we determine various further properties of the components of $B_{n}$-generalized pseudo-K\"{a}hler structures as well
as a rescaling property for such structures (see Corollary \ref{rescaling-1} and  Corollary \ref{rescaling-2}). 
Finally, we remark that a
generalized K\"{a}hler structure on an exact Courant algebroid   may be interpreted as a 
$B_{n}$-generalized K\"{a}hler structure (see Remark \ref{2-even}).

In Section \ref{3-4} we simplify Theorems \ref{integr-odd-introd} and \ref{integr-even-introd} under the assumption that 
$M$ has small dimension. This section is intended as a preparatory material for Section \ref{examples}, 
which is devoted to the classification of left-invariant
$B_{n}$-generalized pseudo-K\"{a}hler structures $(\mathcal{G}, \mathcal F)$  over Lie groups of dimension two, three or four.
Let $G$ be a Lie group with Lie algebra $\mathfrak{g}$ 
and  $E $ a Courant algebroid of type $B_{n}$ over $G$, whose Dorfman bracket is twisted by left-invariant forms
$(H, F)$. A $B_{n}$-generalized pseudo-K\"{a}hler structure on $E$  is called {\cmssl left-invariant} if 
its components are left-invariant tensor fields on $G$ (in particular, the function $c_{+}$ is constant).   
Our main results from Section \ref{examples} can be roughly summarized as follows (below $c_{+}$ and $g$ are part of the components of the corresponding $B_{n}$-generalized 
pseudo-K\"{a}hler structure).

 \begin{thm}\label{summary-lie}  i) Assume that $G$ is $2$-dimensional. There is a left-invariant $B_{2}$-generalized pseudo-K\"{a}hler structure on $E$, such that
 $c_{\pm} \notin \{ -1, +1\}$  if and only if 
 $\mathfrak{g}$ is abelian.\
 
 ii) Assume that $G$ is  $3$-dimensional,  unimodular, with canonical operator $L$  (defined in (\ref{def-L})) diagonalizable.   There is a left-invariant $B_{3}$-generalized pseudo-K\"{a}hler structure on $E$ if and only if  $\mathfrak{g}$ is abelian, 
  $\mathfrak{g}= \mathfrak{so} (2) \ltimes \mathbb{R}^{2} $ or  $\mathfrak{g} = \mathfrak{so}(1,1) \ltimes
 \mathbb{R}^{2}$.\
 
 iii) Assume that $G$ is  $3$-dimensional  and non-unimodular. Let  $\mathfrak{g}_{0}$ be the unimodular kernel of $\mathfrak{g}$. 
 There is a left-invariant $B_{3}$-generalized pseudo-K\"{a}hler structure on 
 $E$ such that $g\vert_{\mathfrak{g}_{0} \times \mathfrak{g}_{0}}$ is non-degenerate if and only if 
$\mathfrak{g} = \mathbb{R} \oplus \mathfrak{sol}_{2}$, where $\mathfrak{sol}_{2}$ is  the unique non-abelian $2$-dimensional Lie algebra.\ 
  
 iv) Assume that $G$ is $4$-dimensional and  
 its unimodular kernel is non-abelian.
 There is an adapted  (see Definition \ref{adapted}) $B_{4}$-generalized pseudo-K\"{a}hler structure on $E$ such that $c_{+} \notin \{ -1, 0, 1\}$ if and only if 
 $\mathfrak{g}$ has a basis $\{ u, e_{1}, e_{2}, e_{3} \}$ in which the Lie bracket takes the form
 \begin{align}
\nonumber& [e_{1}, e_{2} ] = \epsilon_{3} \lambda   e_{3},\ [e_{2}, e_{3} ] =  0, \ [e_{3}, e_{1} ] =\epsilon_{2}  \lambda e_{2},\\
\label{form-1-rez}& [u, e_{1} ] =0,\ [ u, e_{2}  ]=  \beta  e_{3},\ [ u, e_{3} ] = - \beta  e_{2},
\end{align} 
where  
$\lambda \in \mathbb{R}\setminus \{ 0\}$, $\beta  \in \mathbb{R}$,   $\epsilon_{i} \in \{ \pm 1\}$.\ 

In all cases above $E$ is untwisted (i.e. $H=0$, $F=0$).
 \end{thm}

The  description of the left-invariant $B_{n}$-generalized pseudo-K\"{a}hler structures  which arise 
in Theorem \ref{summary-lie}  can be found in Section \ref{sect-dim-2}, Proposition \ref{unimod-3}, Proposition \ref{non-mod-3} and 
Proposition \ref{4c-n0}.  Cases ii) and iii) of Theorem \ref{summary-lie} include a  description of  
all left invariant $B_{3}$-generalized K\"{a}hler structures on Courant algebroids of type $B_{3}$ over $3$-dimensional Lie groups
(see Corollary \ref{classification}).  
The various additional assumptions from Theorem \ref{summary-lie}
were intended to simplify the computations. We finally remark that all examples provided by Theorem \ref{summary-lie} live on untwisted
Courant algebroids.  It would be interesting to find twisted examples.\\

{\bf Acknowledgements.}   Research of V.C.\ was supported by the German Science Foundation (DFG) under Germany's Excellence Strategy  --  EXC 2121 ``Quantum Universe'' -- 390833306. L.D. was partially supported by the UEFISCDI research grant PN-III-P4-ID-PCE-2020-0794, project title  ``Spectral Methods in Hyperbolic Geometry''.

\section{Preliminary material}\label{preliminary}

For completeness of our exposition we recall the definition of a Courant algebroid (the axioms C2 and C3 are redundant and are only included for convenience).

\begin{defn} A  {\cmssl Courant algebroid}  on a manifold $M$ is a vector bundle $E\rightarrow M$ equipped with a non-degenerate symmetric
bilinear form $\langle \cdot , \cdot \rangle\in \Gamma ( \mathrm{Sym}^{2} (E^{*}))$ (called the {\cmssl scalar product}), a bilinear operation 
$[\cdot , \cdot ]$ (called the {\cmssl Dorfman bracket}) on the space of smooth sections $\Gamma (E)$ of $E$ and a homomorphism
of vector bundles  $\pi : E \rightarrow TM$ (called the {\cmssl anchor}) such that the following conditions are satisfied:
for all $u, v, w\in \Gamma (E)$ and $f\in C^{\infty}(M)$,\

C1) $[ u, [v, w]] = [[ u, v], w] + [v, [u, w]]$;\

C2) $\pi ( [u, v]) = [ \pi (u), \pi (v) ]$;\

C3) $[u, fv] = \pi (u)(f) v + f [u, v]$;\

C4) $\pi (u) \langle v, w\rangle = \langle [u, v], w\rangle + \langle v, [u , w]\rangle$;\

C5) $2\langle [u, u], v\rangle = \pi (v) \langle u, u\rangle .$
\end{defn}

According to \cite{rubio}, a  {\cmssl Courant algebroid of type $B_{n}$}  over a manifold $M$ is the vector bundle
$E = TM \oplus T^{*}M\oplus \mathbb{R}$, with scalar product 
\begin{equation}\label{scalar}
\langle X + \xi + \lambda , Y +\eta + \mu \rangle =\frac{1}{2} \left( \eta (X) +\xi (Y)\right) + \lambda \mu ,
\end{equation}
for any $X, Y\in TM$, $\xi , \eta \in T^{*}M$ and $\lambda , \mu \in \mathbb{R}$, anchor  given by the natural projection 
$\pi : E \rightarrow TM$  and 
Dorfman bracket given by
\begin{align}
\nonumber [ X + \xi + \lambda , Y + \eta + \mu ] & = {\mathcal L}_{X} Y + {\mathcal L}_{X}\eta - i_{Y} d\xi +
2\mu d\lambda + i_{X} i_{Y} H\\
\label{dorfmann}  & - 2\left( \mu i_{X} F -\lambda i_{Y} F\right) + X(\mu ) - Y(\lambda  ) + F(X, Y)
\end{align}
for any $X, Y\in {\mathfrak X}(M)$, $\xi , \eta \in \Omega^{1}(M)$ and $\lambda , \mu \in C^{\infty}(M)$,
where $F\in \Omega^{2}(M)$ is a closed $2$-form and $H\in \Omega^{3}(M)$  such that $ dH = - F\wedge F$
(we follow the convention of \cite{rubio}, which differs from the convention from our previous works
\cite{moscow} and \cite{t-duality-paper}
by a minus sign in  the $3$-form $H$).  
We will refer to (\ref{dorfmann}) as  the {\cmssl  Dorfman bracket twisted by $(H, F)$}.  
The Courant algebroid of type $B_{n}$ with Dorfman bracket twisted by $(H, F)$ will be denoted by
$E_{H, F}$ and will be called the {\cmssl Courant algebroid twisted by $(H, F).$} 
An {\cmssl odd exact Courant
algebroid} is a Courant algebroid isomorphic to a Courant algebroid of type $B_{n}.$

\subsection{Generalized metrics}

A {\cmssl generalized  metric} on an odd exact  Courant algebroid $E$,  with scalar product
$\langle \cdot , \cdot \rangle$ and anchor $\pi : E \rightarrow TM$,  
is a subbundle $E_{-}\subset E$ of rank $n :=\mathrm{dim}\, M$ 
on which 
$\langle \cdot , \cdot \rangle$ is non-degenerate  (see  e.g.\ \cite{garcia}).
According to \cite{baraglia}, when  the restriction $\pi\vert_{E_{-}} : E_{-} \rightarrow TM$ 
is an isomorphism, the  generalized metric is called {\cmssl admissible}. 
Any rank $n$  subbundle $E_{-}\subset E$ with the property that the restriction  
$\langle \cdot , \cdot\rangle\vert_{E_{-}}$ is negative definite 
is an admissible  generalized metric.  Such generalized metrics will be called
{\cmssl Riemannian}. 
In this paper we only consider admissible  metrics. For this reason, the word
`admissible' will be omitted.

A generalized metric 
$E_{-}$ on an odd exact Courant algebroid  $E$  defines an 
orthogonal isomorphism 
$\mathcal{G}^{\mathrm{end}}\in \mathrm{End}\, (E)$ by $\mathcal{G}^{\mathrm{end}}\vert_{E_{\pm}} = \pm \mathrm{Id}$,
where $E_{+} : = E_{-}^{\perp}$  (unless otherwise stated, $\perp$ will always denote the orthogonal complement with respect
to the scalar product of the Courant algebroid).
The bilinear form 
\begin{equation}\label{g-end}
\mathcal{G}(u, v) := \langle \mathcal{G}^{\mathrm{end}} u, v\rangle ,\ u, v\in E
\end{equation}
 is given by 
\begin{equation}
\mathcal{G} = \langle \cdot , \cdot\rangle \vert_{E_{+}} - \langle \cdot , \cdot\rangle \vert_{E_{-}}.
\end{equation} 
It  is symmetric and, when $E_{-}$ is a generalized Riemannian metric, 
it is positive definite. 
When  referring to a generalized metric we will freely name either the bundle $E_{-}$ or the bilinear form
$\mathcal{G}$.\

Let $E_{-}$ be a generalized metric on a Courant algebroid $E$ of type $B_{n}$ 
and $E_{+} := E_{-}^{\perp}$. Since   
the restriction of the anchor $\pi\vert_{E_{-}} : E_{-} \rightarrow TM$ is a bijection there is an induced pseudo-Riemannian metric on $M$, defined by
\begin{equation}\label{def-g}
g(X, Y) := - \langle s(X), s(Y)\rangle ,  \forall X, Y\in TM 
\end{equation}
where $s: TM \rightarrow E_{-}$ is the inverse of $\pi\vert_{E_{-}}$.  
When  $E_{-}$ is a generalized Riemannian metric, $g$ is positive definite. 
As for generalized metrics on heterotic Courant algebroids (see \cite{baraglia}), the vector bundles $E_{\pm }$ 
are given by
\begin{align}
\nonumber& E_{-} = \{ X - i_{X} (g-b) - A(X) A + A(X) \mid X\in TM\} \\
\label{form}& E_{+} = \{ X + i_{X}(g + b) + ( A(X) -2\mu)A + \mu \mid    X\in TM,\ \mu \in \mathbb{R}\} ,
\end{align}
where $A$ is a one-form and $b$ is a two-form.

The next lemma was proved in \cite{baraglia} for heterotic Courant algebroids.
The same argument applies  to Courant algebroids of type $B_{n}.$ 
For completeness of our exposition we include its proof.

\begin{lem}\label{metric-standard} 
Let $E_{-}$ be a generalized metric on a Courant algebroid  $E$ of type $B_{n}.$ 
There is an isomorphism between $E$ and another Courant algebroid  $\tilde{E}$  of type $B_{n}$  
which
maps $E_{-}$ to a generalized  metric on $\tilde{E}$ of the form (\ref{form}) with $ A=0$ and $b =0.$
\end{lem}

\begin{proof} Let   $\tilde{E}:= E_{\tilde{H}, \tilde{F}}$ be the Courant algebroid   twisted by $(\tilde{H}, \tilde{F})$, where 
$$
\tilde{H}:= H - db - (2F + dA)\wedge A,\ \tilde{F}:= F + dA
$$
(note that $d\tilde{H} = - \tilde{F}\wedge \tilde{F}$ since 
$dH = - F\wedge F$).
The map $I : E\rightarrow \tilde{E}$ defined by
$$
I(X) = X - A(X) - i_{X}b - A(X) A,\ I(\eta ) = \eta , I(\mu ) = 2 \mu  A + \mu ,
$$
for $X\in {\mathfrak X}(M)$, $\eta \in \Omega^{1}(M)$ and $\mu \in C^{\infty}(M)$ is an isomorphism of Courant algebroids
which  maps $E_{-}$ to the  generalized   metric 
\[ 
\tilde{E}_{-}= \{ X - i_{X} g,\ X\in TM\} . \qedhere 
\]
\end{proof}

Owing to the above lemma  we will  often  assume 
that a generalized metric on a Courant algebroid
of type $B_{n}$ 
 is given by a subbundle $E_{-}$ like in (\ref{form}) 
with
$g$ a pseudo-Riemannian metric,  $b=0$ and $A=0.$ 
Such a generalized metric is  in the {\cmssl  standard form}.
For more details on generalized metrics on arbitrary Courant algebroids, see  e.g.\ \cite{garcia}.

\subsection{$B_{n}$-generalized complex structures}

Following \cite{rubio}, we  recall that a {\cmssl  $B_{n}$-generalized almost complex structure}  on an odd exact  Courant algebroid $(E, 
\langle \cdot , \cdot \rangle , [\cdot , \cdot ], \pi )$ over a manifold $M$ 
of dimension $n$ is a 
complex isotropic  rank $n$ subbundle $L \subset  E_{\mathbb{C}}$ such that $L \cap \bar{L} =0$.
The orthogonal complement  $U_{\mathbb{C}}$ of  $L \oplus \bar{L}\subset E_{\mathbb{C}}$ has   rank  one. It is generated by a section $u_{0} \in \Gamma (E)$ (unique up to multiplication by $\pm 1$)
which satisfies $\langle u_{0}, u_{0} \rangle = (-1)^{n}$.
Let $\mathcal F\in\Gamma (  \mathrm{End}\,  E)$ be the  endomorphism 
with $i$-eigenbundle $L$,  $(-i)$-eigenbundle $\bar{L}$ and $\mathcal F u_{0} =0.$   
It is $\langle \cdot , \cdot \rangle$-skew-symmetric and  satisfies 
\begin{equation}\label{endo-f}
\mathcal F^{2} = -\mathrm{Id}
+ (-1)^{n}  \langle \cdot , u_{0} \rangle u_{0}.
\end{equation}
We  will often call  $\mathcal F$  
(rather than $L$) 
a $B_{n}$-generalized almost complex structure.
We say that $\mathcal F$  is a {\cmssl $B_{n}$-generalized complex structure}
(or is {\cmssl integrable}) 
if 
$L$ is  integrable. (A subbundle of $E$ or its complexification $E_{\mathbb{C}}$ is called {\cmssl integrable} if its space of sections is closed
under the  Dorfman bracket).   Then  $L$ becomes a Lie algebroid with Lie bracket induced by the  Dorfman bracket of 
$E$.\

\section{Definition of $B_{n}$-generalized pseudo-K\"{a}hler structures  }\label{definitions}

Let $(E, \langle \cdot , \cdot \rangle , [\cdot , \cdot ] , \pi )$ be an odd exact Courant algebroid  over a manifold $M$  of dimension $n$.

\begin{defn} A {\cmssl $B_{n}$-generalized almost pseudo-Hermitian   structure}  $(\mathcal{G}, \mathcal F)$ on $E$ is a 
generalized metric $\mathcal{G}$  together with a $B_{n}$-generalized almost complex structure
 $\mathcal F$ such that $\mathcal{G}^{\mathrm{end}} \mathcal F = \mathcal F \mathcal{G}^{\mathrm{end}}.$
\end{defn}

The commutativity condition $\mathcal{G}^{\mathrm{end}} \mathcal F = \mathcal F \mathcal{G}^{\mathrm{end}}$ implies  that
$\mathcal F$ preserves the bundles $E_{\pm}$ determined by the generalized metric. 
The next lemma summarizes some simple properties of $B_{n}$-generalized almost pseudo-Hermitian   structures.

\begin{lem} \label{neu} Let $(\mathcal{G}, \mathcal F)$ be a $B_{n}$-generalized almost pseudo-Hermitian   structure on $E$,
$U:= \mathrm{Ker}\, \mathcal F$ 
 and 
$u_{0} \in\Gamma ( U)$  such that  $\langle u_{0}, u_{0} \rangle = (-1)^{n}.$ 
The following statements hold:

i)  $u_{0}$ is a section of $E_{+}$ if $n$ is even and a section of $E_{-}$ is $n$ is odd, i.e.
\begin{equation}\label{def-gend}
\mathcal{G}^{\mathrm{end}} (u_{0} ) =  (-1)^{n} u_{0}.
\end{equation}

ii)  For any $u, v\in E$, 
\begin{align}
\nonumber& \mathcal{G}(\mathcal F u, \mathcal F v) = \mathcal{G}(u, v) -\langle u, u_{0} \rangle \langle v, u_{0}\rangle ;\\
\label{prop-G} & \mathcal{G}(\mathcal F u, v) = - \mathcal{G}(u, \mathcal F v) .
 \end{align}
  \end{lem}

\begin{proof} i) Since $U$ has  rank one and $\mathcal{G}^{\mathrm{end}}\mathcal F = \mathcal F \mathcal{G}^{\mathrm{end}}$, we obtain that 
$\mathcal{G}^{\mathrm{end}} (u_{0}) = \lambda u_{0}$ for $\lambda \in  C^{\infty}(M)$.   From $(\mathcal{G}^{\mathrm{end}})^{2} =
\mathrm{Id}$ we obtain that $\lambda = \pm 1$,
  i.e.\ $u_{0} \in\Gamma ( E_{+})$ or $u_{0} \in \Gamma (E_{-}).$
 On the other hand, $\mathcal F$ is a complex structure on $u_{0}^{\perp}$ and preserves $E_{\pm}$.    
 In particular, $\mathcal F$ is a complex structure on  $u_{0}^{\perp}\cap E_{\pm}$. 
If  $u_{0}\in E_{+}$, then  $E_{-} = E_{-} \cap  u_{0}^{\perp}$ 
(because $E_{\pm}$ are orthogonal)
and  $\mathcal F\vert_{E_{-}}$ is a complex structure  on $E_{-}$. This  implies that $n = \mathrm{rank}\, E_{-} $ is even.   If  $u_{0}\in E_{-}$, then  
$$
E_{-} =  ( E_{-} \cap u_{0}^{\perp} ) \oplus \mathrm{span} \{ u_{0} \} 
$$
and $\mathcal F$ restricts to a complex structure on  $E_{-} \cap u_{0}^{\perp}$.  We obtain that $n$ is odd.\

ii) Relations  (\ref{prop-G}) 
are consequences of the properties of $\mathcal{G}^{\mathrm{end}}$ and $\mathcal F .$ 
\end{proof}

In the setting of Lemma \ref{neu},  let $\mathcal F_{2} := \mathcal{G}^{\mathrm{end}} \mathcal F .$ Then  $\mathcal F_{2}$ 
is skew-symmetric with respect to $\langle \cdot , \cdot \rangle $,  satisfies relation   (\ref{endo-f})
and $\mathrm{Ker}\, \mathcal F_{2}=  U$.   Its $i$-eigenbundle is given by
$$
L_{2} := L_{1} \cap (E_{+})_{\mathbb{C}} \oplus \bar{L}_{1} \cap (E_{-})_{\mathbb{C}} 
$$ 
where $L_{1}$ is the $i$-eigenbundle of $\mathcal F .$ In particular, $\mathrm{rank}\, L_{2} =\mathrm{rank}\, L_{1} =n.$ 
We deduce that  $\mathcal F_{2}$ is a $B_{n}$-generalized almost
complex structure. We obtain the following  alternative definition of $B_{n}$-generalized almost pseudo-Hermitian structures
on  odd exact Courant algebroids.

\begin{prop} A  $B_{n}$-generalized almost  pseudo-Hermitian  structure on $E$ is equivalent to a pair $(\mathcal F_{1}, \mathcal F_{2})$ of commuting $B_{n}$-generalized
almost complex structures such that 
the bilinear form 
 $(u, v)\mapsto \langle \mathcal F_{1} \mathcal F_{2}  (u ),  v\rangle $
on  $U^{\perp}$ is nondegenerate 
and  
\begin{equation}\label{m1}
\pi :   \{ u\in U^{\perp}\mid  \mathcal F_{1} (u) =-  \mathcal F_{2} (u) \} \rightarrow TM
\end{equation} 
is an isomorphism when $n$ is even while
\begin{equation}\label{m2}
\pi : \{  u\in E  \mid  \mathcal F_{1} (u) =-  \mathcal F_{2} (u) \} \rightarrow TM
\end{equation} 
is an isomorphism when $n$ is odd.
Above 
 $U= \mathrm{Ker}\,  \mathcal F_{1} = \mathrm{Ker}\, \mathcal F_{2}$.
\end{prop}

\begin{proof} Given $(\mathcal F_{1}, \mathcal F_{2})$ as in the statement of the proposition, 
we recover the  generalized metric  $\mathcal{G}$ from   $\mathcal{G}^{\mathrm{end}} :=  -\mathcal F_{1} \mathcal F_{2}$ on $U^{\perp}$ and
$\mathcal{G}^{\mathrm{end}} = (-1)^{n} \mathrm{Id}$ on $U$. The requirement that (\ref{m1}) and (\ref{m2}) are isomorphisms
is equivalent to the fact that $\mathcal{G}$ is admissible.
\end{proof}

\begin{rem}{\rm In the setting of the above proposition, if  $ - \langle  \mathcal F_{1} \mathcal F_{2} u, u\rangle >0$ for any $u\in U^{\perp}\setminus \{ 0\}$ then
$\mathcal{G}$ is a generalized Riemannian metric.}
\end{rem}

\begin{defn} A {\cmssl $B_{n}$-generalized  pseudo-K\"{a}hler 
structure}  is a $B_{n}$-generalized almost pseudo-Hermitian structure $(\mathcal{G}, \mathcal F)$ for which 
$\mathcal F$  and $\mathcal{G}^{\mathrm{end}}\mathcal F$  are $B_{n}$-generalized complex structures. 
\end{defn}

For  $u,v\in \Gamma (E)$  and $A\in\Gamma (\mathrm{End} \, E)$ we define ${\mathbf L}_{u} v := [u,v]$ and 
${\mathbf L}_{u}A\in\Gamma
( \mathrm{End}\, E)$ by
$$
({\mathbf L}_{u} A)(v) := [u, Av ] - A [u, v],\ \forall v\in \Gamma (E)
$$
and  refer to it as the {\cmssl Dorfman Lie derivative of $A$ in the direction of $u$}.
It is easy to check that ${\mathbf L}_{u}$ acts as a derivation on the 
$\mathbb{R}$-algebra $\Gamma (\mathrm{End} \, E)$.

\begin{lem}\label{dorf-u0}   Let $(\mathcal{G},   \mathcal F )$ be a $B_{n}$-generalized pseudo-K\"{a}hler structure and $u_{0} \in\Gamma ( \mathrm{Ker}\, \mathcal F )$ with 
$\langle u_{0}, u_{0}\rangle = (-1)^{n}.$ Then  ${\mathbf L}_{u_{0}}\mathcal F =0$
and ${\mathbf L}_{u_{0}} \mathcal{G}^{\mathrm{end}} =0.$
\end{lem}

\begin{proof} The statement ${\mathbf L}_{u_{0} }\mathcal F =0$ for any $B_{n}$-generalized complex structure $\mathcal F$
on an odd exact Courant algebroid 
and $u_{0}\in\Gamma ( \mathrm{Ker}\, \mathcal F )$ normalized such that $\langle u_{0} , u_{0} \rangle = (-1)^{n}$ was proved in
Lemma 4.13 of \cite{rubio}.
Applying this statement to
 $\mathcal F_{2} := \mathcal{G}^{\mathrm{end}} \mathcal F$ we obtain 
${\mathbf L}_{u_{0} }\mathcal F_{2} =0$. 
From
 ${\mathbf L}_{u_{0}}(u_{0} ) = \frac{1}{2} \pi^{*} d\langle u_{0}, u_{0} \rangle =0$ we deduce  that 
${\mathbf L}_{u_{0}}$ preserves $u_{0}^{\perp}$.
Combined with  $\mathcal{G}^{\mathrm{end}} = - \mathcal F  \mathcal F_{2}$ on $u_{0}^{\perp}$ we obtain  ${\mathbf L}_{u_{0}}\mathcal{G}^{\mathrm{end}} =0.$
\end{proof}

\section{Components of $B_{n}$-generalized almost pseudo-Hermitian structures}\label{components-hermitian}

In this section we describe $B_{n}$-generalized almost pseudo-Hermitian structures  $(\mathcal{G}, \mathcal F)$ on a
Courant algebroid  $E:= E_{H, F}$ of type $B_{n}$ over a manifold $M$ of dimension $n$
  in terms of tensor fields on $M$. 
We assume that  the generalized metric $\mathcal{G}$ is in the standard form.
Let $E:= E_{+} \oplus E_{-}$ be the decomposition  of $E$ determined by $\mathcal{G}$.
From the description (\ref{form}) of $E_{\pm}$ 
(with $b=0$, $A=0$) we obtain  canonical  isomorphisms 
\begin{align}
\nonumber&  E_{-}\cong TM,\  X -  i_{X} g\mapsto X,\\
\label{iso-can} &  E_{+} \cong TM \oplus
\mathbb{R},\ X + i_{X} g+ a\mapsto X +a 
\end{align}
where $g$ is the  pseudo-Riemannian metric on $M$ induced by $\mathcal{G}$, see
(\ref{def-g}). The second isomorphism (\ref{iso-can}) maps $\langle\cdot , \cdot \rangle\vert_{E_{+}}$
to the metric
$$
(g+ g_{\mathrm{can}})(X + \lambda, Y+ \mu ) = g(X, Y) + \lambda\mu,\  \forall X, Y\in TM,\ \lambda , \mu\in\mathbb{R}
$$ 
on $TM \oplus \mathbb{R}$, where $g_{\mathrm{can}} ( \lambda , \mu ) = \lambda \mu .$  
Let $u_{0} \in\Gamma ( \mathrm{Ker}\,  \mathcal F )$, normalized such that $\langle u_{0}, u_{0} \rangle = (-1)^{n}$.
When $n$ is even, $u_{0}\in \Gamma (E_{+}) $ and  will be denoted by $u_{+}.$ When $n$ is odd, 
$u_{0} \in \Gamma (E_{-})$ and 
will be denoted by $u_{-}.$ (See Lemma~\ref{neu}.)

In the next proposition   `$\perp$'  denotes the orthogonal complement in $TM$ with respect 
to  $g$.

\begin{prop}\label{data-on-M} i) Assume that $n$ is odd.  A $B_{n}$-generalized  almost  pseudo-Hermitian  structure 
$(\mathcal{G}, \mathcal F)$ on $E$
is equivalent to the data  $(g, J_{+}, J_{-}, X_{+}, X_{-})$ 
where $g$ is a pseudo-Riemannian metric on $M$, 
$J_{\pm}\in \Gamma ( \mathrm{End}\, TM)$ are\linebreak[4]  $g$-skew-symmetric endomorphisms and 
$X_{\pm}\in {\mathfrak X}(M)$ are vector fields,   such that  $J_{\pm } X_{\pm} =0$,  
$J_{\pm}\vert_{ X_{\pm}^{\perp}}$ are complex structures and $g(X_{\pm}, X_{\pm}) =1.$ \

ii) Assume that $n$ is even. A $B_{n}$-generalized almost pseudo-Hermitian structure 
$(\mathcal{G}, \mathcal F)$ 
on $E$
is equivalent to the data $(g,  J_{+},  J_{-}, X_{+} , X_{-},  c_{+})$ where $g$ is a pseudo-Riemannian  metric on $M$,
$J_{\pm}\in \Gamma ( \mathrm{End}\, TM)$ are  $g$-skew-symmetric endomorphisms, 
$X_{\pm}\in {\mathfrak X}(M)$  are $g$-orthogonal vector fields and
$c_{+}\in C^{\infty}(M)$ is a function, such that
$J_{-}$ is an almost complex structure on $M$, 
$J_{+}$ satisfies 
\begin{align}
\nonumber&  J_{+}X_{+} = - c_{+} X_{-},\ J_{+}X_{-} = c_{+}X_{+},\\
\nonumber& J_{+}^{2} X = - X + g(X, X_{+}) X_{+} + g(X, X_{-}) X_{-},\ \forall X\in TM
\end{align}
and $g(X_{+}, X_{+}) =
g(X_{-}, X_{-})= 1-c_{+}^{2}$.

iii) In both cases ($\mathrm{dim}\, M$ odd or even), the generalized 
metric $\mathcal{G}$ given by
\begin{equation}\label{g}
E_{-}= \{  X -  i_{X} g\mid X\in TM\}.
\end{equation}
If $\mathrm{dim}\, M$ is odd/even, the $B_{n}$-generalized almost complex structures $\mathcal F_{1} = \mathcal F$ 
and $\mathcal F_{2} = \mathcal{G}^{\mathrm{end}} \mathcal F$ are given by 
\begin{equation}\label{components-odd}
\mathcal F = \left(
\begin{tabular}{ccc}
$\frac{1}{2} (J_{+}+ J_{-}) $ & $ \frac{1}{2}(J_{+} - J_{-}) \circ g^{-1}$ & $X_{\pm}$\\
$\frac{1}{2} g \circ (J_{+} - J_{-})$ & $-\frac{1}{2} (J_{+} + J_{-})^{*}$ & $i_{X_{\pm}} g$\\
$-\frac{1}{2} i_{X_{\pm}} g$ &  $-\frac{1}{2} X_{\pm}$ & $0$
\end{tabular}\right)
\end{equation}
and
\begin{equation}\label{components-odd_2}
\mathcal F_{2} = \left(
\begin{tabular}{ccc}
$\frac{1}{2} (J_{+}- J_{-}) $ & $ \frac{1}{2}(J_{+} + J_{-}) \circ g^{-1}$ & $X_{\pm}$\\
$\frac{1}{2} g \circ (J_{+} + J_{-})$ & $-\frac{1}{2} (J_{+} - J_{-})^{*}$ & $i_{X_{\pm}} g$\\
$-\frac{1}{2} i_{X_{\pm}}g$ &  $-\frac{1}{2} X_{\pm}$ & $0$
\end{tabular}\right) .
\end{equation}
\end{prop}

\begin{proof} i)  
The scalar product $-\langle \cdot , \cdot
\rangle\vert_{E_{-}}$ corresponds to $g$  
under the first isomorphism (\ref{iso-can}). 
Then  $X_{-}:= \pi (u_{-})$  satisfies $g(X_{-}, X_{-})=1$,  since
$\langle u_{-}, u_{-} \rangle =-1.$
By means of the first isomorphism (\ref{iso-can}), 
the restriction   $\mathcal F\vert_{E_{-}}$ induces a $g$-skew-symmetric  endomorphism $J_{-}$ of $TM$,
with $J_{-}X_{-} =0$ and   which is
a  complex structure on $X_{-}^{\perp}$. 
Consider now the second isomorphism (\ref{iso-can}). By means of this isomorphism, 
$\langle \cdot , \cdot \rangle\vert_{E_{+}}$ corresponds  to 
$g+ g_{\mathrm{can}}$ 
on $TM\oplus \mathbb{R}$ 
and
$\mathcal F\vert_{E_{+}}$ induces a complex structure ${\mathcal J}_{+}$ on $TM\oplus \mathbb{R}$, 
skew-symmetric with respect to $g+ g_{\mathrm{can}}$. An easy computation shows that ${\mathcal J}_{+}$ is of the form 
$$
{\mathcal J}_{+} =\left(
\begin{tabular}{cc}
$J_{+}$ & $ X_{+}$\\
$-  i_{X_{+}} g$ & $ 0$
\end{tabular}\right) ,
$$
where $X_{+}\in {\mathfrak X}(M)$ satisfies $g(X_{+}, X_{+})=1$ and  $J_{+}\in\Gamma (  \mathrm{End}\, TM)$  is\linebreak[4] $g$-skew-symmetric, $J_{+}X_{+}=0$ and
$J_{+}\vert_{ X_{+}^{\perp}}$ is a complex structure on  $X_{+}^{\perp}.$

ii) When $n$ is even,  $\mathcal F\vert_{E_{-}}$ 
induces, under the 
first isomorphism (\ref{iso-can}),   a
$g$-skew-symmetric 
 almost complex structure $J_{-}$ on $M$.
Write  $u_{+} = X_{+} + g(X_{+}) + c_{+}$, 
where $X_{+}\in {\mathfrak X}(M)$ and $c_{+} \in C^{\infty}(M).$
From $\langle u_{+}, u_{+}\rangle =1$ we obtain 
$g(X_{+}, X_{+}) = 1- c_{+}^{2}.$
Under the second  isomorphism 
(\ref{iso-can}),   the restriction
$\mathcal F\vert_{E_{+}}$ induces an endomorphism 
$\mathcal F_{+}\in \Gamma ( \mathrm{End} (TM\oplus \mathbb{R}))$, 
which satisfies 
$$
{\mathcal F}_{+}(X_{+}+c_{+} ) =0,\ \mathcal F_{+}^{2} = -\mathrm{Id} + i_{X_{+} + c_{+}}(g + g_{\mathrm{can}} )\otimes 
(X_{+} + c_{+})
$$
and is skew-symmetric with respect to $g + g_{\mathrm{can}}$.   
Writing $\mathcal F_{+}$ in block form 
$$
{\mathcal F}_{+} =\left(
\begin{tabular}{cc}
$J_{+}$ & $ X_{-}$\\
$\omega$ & $ a$
\end{tabular}\right)
$$
where $J_{+}\in\Gamma (  \mathrm{End}\, TM)$, $X_{-}\in {\mathfrak X}(M)$, $\omega\in \Omega^{1}(M)$ and
$a\in C^{\infty}(M)$ 
an easy check shows  that  
$a=0$,  $\omega = -  i_{X_{-}} g$,  $g(X_{-}, X_{-}) = 1 - c_{+}^{2}$, 
$ g(X_{+}, X_{-})=0$ and $J_{+}$ satisfies the required properties.\

iii) Relations (\ref{components-odd}) and (\ref{components-odd_2}) follow from i) and ii). 
For instance, to compute $\mathcal{F}$ on a element $X\in TM$ it suffices to decompose
$X = \frac12 (X+ i_{X} g) +\frac12 (X- i_{X} g)\in E_+\oplus E_-$ and to apply the above formulas for $\mathcal{F}|_{E_\pm}$. 
This yields the first column of (\ref{components-odd}). The second column is obtained similarly after decomposing an element $\xi\in T^*M$ as $\xi = \frac12 (X+i_{X} g) -\frac12 (X- i_{X} g)$, where $X=g^{-1}\xi$.
\end{proof}

\begin{defn}
Let $(\mathcal{G}, \mathcal F)$ be a $B_{n}$-generalized pseudo-Hermitian structure on a Courant algebroid $E_{H, F}$ of type
$B_{n}$ over a manifold $M$.  The tensor fields on $M$  constructed in  Proposition \ref{data-on-M}, are called 
the {\cmssl{components}}  of $(\mathcal{G}, \mathcal F ).$  
\end{defn}

\section{Components of $B_{n}$-generalized pseudo-K\"{a}hler structures}\label{components-kahler}

In this section we express the integrability of a $B_{n}$-generalized almost pseudo-Hermitian structure on a Courant
algebroid   $E$  of type $B_{n}$ in terms of its components. This is done in  Theorems \ref{integrability-odd} and 
\ref{1-even}  from the next section. 
The proofs of  these theorems will be   presented in Section
\ref{proofs-section}.

\subsection{Statement of results}\label{integrability-statements}

Let  $E= E_{H, F}$ be a Courant algebroid of type $B_{n}$ 
over a manifold $M$ of dimension $n$, with  Dorfman bracket twisted   by $(F, H)$, 
and
$(\mathcal{G}, \mathcal F )$ a $B_{n}$-generalized almost  pseudo-Hermitian structure on $E$.

\begin{notation}{\rm i) \label{H-F-notation}
In analogy with the standard notation for the $(1,0)$-bundle of  an almost
complex structure on a manifold, for any  endomorphism
$A\in \Gamma (\mathrm{End}\, TM)$, we  will denote by 
 $T^{(1,0)}_{A}M\subset (TM)_\mathbb{C}$ its  $i$-eigenbundle.\

ii) Let $g$ be the pseudo-Riemannian metric which is part of the components of $( \mathcal{G}, \mathcal F )$, see 
equation~(\ref{def-g}).
We  identify $TM$ with $T^{*}M$ 
using $g$. 
 For $X, Y\in \mathfrak {X}(M)$ we denote by $F(X)$ and $H(X, Y)$ the vector fields identified with the $1$-forms 
 $i_{X}F$ and $i_{Y} i_{X} H.$
 A  decomposable tensor $A\otimes B\in TM\otimes TM$ with $A, B\in TM$ is identified with  the endomorphism of $TM$ which assigns 
 to $X\in TM$ the vector $g(A, X) B.$ }
 \end{notation}

\subsubsection{The case  of odd $n$}
Assume  that $n$ is odd and let
$( g,   J_{+}, J_{-}, X_{+}, X_{-})$   be the components of 
$(\mathcal{G}, \mathcal  F)$.
Define connections $\nabla^+$ and $\nabla^-$ on $TM$ by
\begin{align}
\nonumber& \nabla^{-} _{X}:= \nabla_{X} +\frac{1}{2} H(X),\\
 \label{nabla-1-def} &{\nabla}^{+}_{X}:=\nabla_{X} -\frac{1}{2} H(X)- J_{+} F(X) \otimes X_{+},
\end{align}
where $X\in \mathfrak{X}(M)$, $\nabla$ is the Levi-Civita connection of $g$, 
$H(X)$ denotes the skew-symmetric endomorphism $Y\mapsto H(X,Y)$ and the vector field $J_{+} F(X)$
is metrically identified with a one-form.

\begin{thm}\label{integrability-odd} The $B_{n}$-generalized almost pseudo-Hermitian  structure  
 $(\mathcal{G}, \mathcal F )$ is $B_{n}$-generalized pseudo-K\"{a}hler   if and only if the 
following conditions hold:

i)   The connections  $\nabla^\pm$  preserve the distributions   
$T_{J_\pm}^{(1,0)}M$ (respectively)  and 
\begin{equation}\label{covariant-deriv}
\nabla_{X}^{-} X_{-} =0,\  \nabla^{+}_{X} X_{+} = - J_{+} F(X),\ \forall X\in {\mathfrak X}(M). 
\end{equation} 

ii) The forms $H$ and $F$ satisfy the constraints
\begin{align}
\nonumber& H\vert_{ \Lambda^{3} T_{J_{\pm}}^{(1,0)}M }=0,\   
(i_{X_{+}}H)\vert_{\Lambda^{2}T_{J_{+}}^{(1,0)}M}
= i F\vert_{\Lambda^{2} T_{J_{+}}^{(1,0)}M},\\
\label{rel-F-H}  &
(i_{X_{-}}H)\vert_{\Lambda^{2}T_{J_{-}}^{(1,0)}M}
=0,\  F\vert_{\Lambda^{2} T_{J_{-}}^{(1,0)}M}=0,\
i_{X_{-}} F =0.
\end{align}
\end{thm} 

The above theorem has several immediate consequences.

\begin{cor}\label{rem-comments}In the setting of Theorem \ref{integrability-odd},  if 
$(\mathcal{G}, \mathcal F )$ is a $B_{n}$-generalized pseudo-K\"{a}hler  structure 
with components $(g,  J_{+}, J_{-}, X_{+}, X_{-})$, 
then
$X_{-}$ is  a Killing  field with respect to $g$ which commutes with $X_{+}$ and  preserves
the endomorphisms $J_{\pm}$.   
\end{cor}

\begin{proof}
From the first relation (\ref{covariant-deriv}) 
and the definition of $\nabla^{-}$ 
we obtain that $X_{-}$ is Killing.
In order to prove that $X_{+}$ and $X_{-}$ commute, we write
\begin{equation}
{\mathcal L}_{X_{-}}X_{+} = \nabla_{X_{-}} X_{+} -\nabla_{X_{+}} X_{-} = \frac{1}{2} H(X_{-}, X_{+}) 
+\frac{1}{2} H(X_{+}, X_{-})=0,  
\end{equation} 
where we used  both relations (\ref{covariant-deriv}) and $i_{X_{-}} F=0$. 
It remains to prove that ${\mathcal L}_{X_{-}} J_{\pm } =0.$ 
For any $X\in \Gamma ( T^{(1,0)}_{J_{-}} M)$,
\begin{equation}
{\mathcal L}_{X_{-}} X = \nabla_{X_{-}} X - \nabla_{X} X_{-} = \nabla_{X_{-}}^{-} X - H(X_{-}, X),
 \end{equation}
where we used again
the definition of $\nabla^{-}$ and  $\nabla_{X} X_{-} = - \frac{1}{2} H(X, X_{-})$, which follows 
from the first relation (\ref{covariant-deriv}). Since $\nabla^{-}$ preserves 
$T^{(1,0)}_{J_{-}}M$ and $(i_{X_{-}} H)\vert_{\Lambda^{2} T^{(1,0)}_{J_{-}} M} =0$, 
we obtain that
${\mathcal L}_{X_{-}} X $ is a section of $T^{(1,0)}_{J_{-}} M$. 
Combining this fact with  $J_{-} X_{-}=0$ we obtain  ${\mathcal L}_{X_{-}} J_{-} =0.$
The claim ${\mathcal L}_{X_{-}} J_{+} =0$ can be checked 
by similar computations.
\end{proof}

There is a rescaling property of $B_{n}$-generalized pseudo-K\"{a}hler structures, which also follows from
Theorem \ref{integrability-odd}. 

\begin{cor}\label{rescaling-1}
Let $(\mathcal{G}, \mathcal F)$ be a $B_{n}$-generalized pseudo-K\"{a}hler structure on a Courant 
algebroid $E_{H, F}$ of type $B_{n}$ over an odd dimensional manifold $M$, with components 
$(g,  J_{+}, J_{-}, X_{+}, X_{-})$.
 For any $\lambda \in \mathbb{R}\setminus \{ 0\}$,  the data 
$( \tilde{g} := \lambda^{2} g,
\tilde{J}_{+} := J_{+},  \tilde{J}_{-} := J_{-}, 
\tilde{X}_{+} := \frac{1}{\lambda } X_{+},
\tilde{X}_{-} := \frac{1}{\lambda } X_{-})$ defines a $B_{n}$-generalized pseudo-K\"{a}hler
structure on   the Courant algebroid $E_{\tilde{H} , \tilde{F}}$, where $\tilde{H} := \lambda^{2} H$ and $\tilde{F} :=
\lambda F.$
\end{cor}   

\begin{proof}Let us denote by $\tilde{\nabla}^\pm$ the connections associated with the 
rescaled data. Since 
\begin{eqnarray*}&& H(X,Y)=g^{-1}H(X,Y,\cdot ) = \tilde{g}^{-1}\tilde{H}(X,Y,\cdot)=\tilde{H}(X,Y),\\
&&F(X,Y)\otimes X_+ = \tilde{F}(X,Y)\otimes \tilde{X}_+,\end{eqnarray*}  
for all $X,Y\in  \mathfrak{X}(M)$ 
and $\tilde{J}_+=J_+$, we see that 
$\tilde{\nabla}^{\pm} = \nabla^{\pm}$. Using this fact, it is easy to check that the 
rescaled data satisfy the conditions from Theorem~\ref{integrability-odd}. 
For instance, we get 
\[  \tilde{\nabla}^{+}_{X} \tilde{X}_{+} = \frac{1}{\lambda}\nabla^{+}_{X} X_{+} =-  \frac{1}{\lambda}J_+F(X)
= -\tilde{J}_+\tilde{F}(X),
\]
where we have used that $\tilde{F}(X)= \tilde{g}^{-1}\tilde{F}(X,\cdot ) = \frac{1}{\lambda}g^{-1}F(X, \cdot ) =  \frac{1}{\lambda} F(X)$. 
\end{proof}

\subsubsection{The case of even $n$}

Assume  that  $n$  is even and let 
$(g,   J_{+}, J_{-}, X_{+}, X_{-}, c_{+})$  be the components  of $( \mathcal{G}, \mathcal F )$. 
To simplify the computations  we assume that $X_{\pm}$ are non-null,  i.e.\   $g( (X_{\pm})_{p}, (X_{\pm})_{p}) \neq 0$ or
$c_{\pm}(p)\neq \pm 1$, for any $p\in M.$
We
define the connections
\begin{align}
\nonumber& D^{+}_{X} := \nabla_{X} -\frac{1}{2} H(X) +\frac{c_{+}}{ 1- c_{+}^{2}} F(X)\otimes X_{+} -\frac{J_{+}
F(X)}{ 1-c_{+}^{2}} \otimes X_{-}\\
\label{d-p-m-nabla}& D^{-}_{X} := \nabla_{X} +\frac{1}{2} H(X)
\end{align}
on $TM$,  where $\nabla$ is the Levi-Civita connection of $g$.

\begin{thm}\label{1-even} 
The $B_{n}$-generalized 
almost pseudo-Hermitian structure $( \mathcal{G}, \mathcal F )$ is $B_{n}$-generalized  pseudo-K\"{a}hler  if and only if the following conditions hold:\

i) $D^{\pm }$ preserve $T^{(1,0)}_{J_{\pm}} M$, the vector fields $\{ X_{+}, X_{-} \}$ commute 
 and  their covariant derivatives  $\nabla X_{\pm}$  are given by
\begin{align}
\nonumber  &\nabla_{X} X_{+} = -\frac{1}{2}( i_{X_{+}} H)(X) + c_{+} F(X);\\
\label{d-p-m-civita} & \nabla_{X} X_{-}= -\frac{1}{2}( i_{X_{-}} H)(X) - J_{+} F(X),
\end{align}
for any $X\in {\mathfrak X}(M)$.\

ii) The forms  $F$ and $H$ satisfy  
\begin{equation}\label{H-F-alg-even}
F\vert_{\Lambda^{2} T_{J_{-}}^{(1,0)}M }=0,\quad  H\vert_{\Lambda^{3} T_{J_{\pm}}^{(1,0)}M }=0,\quad
(i_{X_{+}+ ic_{+} X_{-}}H)\vert_{\Lambda^{2} T_{J_{+}}^{(1,0)}M }= 0,
\end{equation}
and are related by
\begin{equation}\label{H-F-even}
F\vert_{\Lambda^{2} T_{J_{+}}^{(1,0)}M } = - i  (i_{X_{-}}H)\vert_{\Lambda^{2} T_{J_{+}}^{(1,0)}M }.
\end{equation}

iii) The function $c_{+}$ 
is Hamiltonian for the closed $2$-form $F$, with Hamiltonian vector field $X_{+}$:
\begin{equation}\label{c-F}
d c_{+} =  i_{X_{+}} F.
\end{equation}
\end{thm}

In computations we shall often replace the condition $\mathcal L_{X_{+} } X_{-} =0$  from the above theorem with a relation between the forms
$H$ and $F$, as follows.

\begin{lem}\label{exchange} In the setting of Theorem 
\ref{1-even}, the commutativity of $X_{+}$ and $X_{-}$ can be replaced by the relation
\begin{equation}\label{H-F-even-prime}
H(X_{+}, X_{-}) = c_{+}F(X_{-})+ J_{+} F(X_{+})
\end{equation}
\end{lem}

\begin{proof}
The second relation (\ref{d-p-m-civita}) applied to $X:= X_{+}$ implies
\begin{equation}\label{add-j-plus}
J_{+} F(X_{+}) = -\nabla_{X_{+}} X_{-} -\frac{1}{2} H(X_{-}, X_{+}).
\end{equation}
We  obtain that (\ref{H-F-even-prime}) is equivalent to
\begin{equation}\label{preg-lie}
\nabla_{X_{+}} X_{-} = -\frac{1}{2} H(X_{+}, X_{-}) + c_{+} F(X_{-}).
\end{equation}
We conclude  by writing $\mathcal L_{X_{+}} X_{-} = \nabla_{X_{+}} X_{-} - \nabla_{X_{-}} X_{+}$
and combining the first relation (\ref{d-p-m-civita}) with $X:= X_{-}$ with (\ref{preg-lie}). 
\end{proof}

The above theorem has several immediate consequences.

\begin{cor} 
In the setting of Theorem \ref{1-even},  if 
$( \mathcal{G}, \mathcal F )$ is a $B_{n}$-generalized pseudo-K\"{a}hler  structure 
with components $(g,  J_{+}, J_{-}, X_{+}, X_{-},  c_{+})$, then 
$X_{+}$ is a Killing  field   and the  almost complex structure $J_{-}$ is integrable.
\end{cor}

\begin{proof} The first relation  (\ref{d-p-m-civita}) implies  that $X_{+}$ is a Killing field.   
The integrability of $J_{-}$  follows from the proof of
Theorem  \ref{1-even} (see Corollary  \ref{pre-final-2} iii) below).
\end{proof}

Another immediate consequence of Theorem \ref{1-even} is a rescaling property of  certain $B_{n}$-generalized  pseudo-K\"{a}hler
structures defined on Courant algebroids $E_{H, F}$ with  trivial closed $2$-form $F$. Remark that on such
Courant algebroids  any  $B_{n}$-generalized  pseudo-K\"{a}hler
structure has $c_{+}$ constant (owing to  the second relation (\ref{c-F}) above).

\begin{cor}\label{rescaling-2}
Let $(\mathcal{G}, \mathcal F)$ be a $B_{n}$-generalized pseudo-K\"{a}hler structure on a Courant 
algebroid $E_{H}:= E_{H, 0}$ of type $B_{n}$ over an even dimensional manifold $M$.  Let 
$(g,  J_{+}, J_{-}, X_{+}, X_{-},  c_{+})$ be the components of $(\mathcal{G}, \mathcal F)$ and assume  that
$c_{+} \not\in \{ -1, 0, 1\}$. Define $\tilde{g}:= \epsilon g$ where $\epsilon := \mathrm{sign}\, ( 1- c_{+}^{2})$,  
$\tilde{X}_{\pm}:=$\linebreak[4]  $|1- c_{+}^{2}|^{-1/2}  X_{\pm}$,
and  $\tilde{J}_{\pm } \in \mathrm{End}\,  (TM)$ by  
$\tilde{J}_{-} := J_{-}$,  $\tilde{J}_+|_{\{ X_{+}, X_{-}   \}^{\perp}} := J_+|_{\{ X_{+}, X_{-}   \}^{\perp}}$  and $\tilde{J}_{+} X_{+} = \tilde{J}_{+} X_{-} =0$.
Then $( \tilde{g}, \tilde{J}_{+}, \tilde{J}_{-}, \tilde{X}_{+}, \tilde{X}_{-}, \tilde{c}_{+}:= 0)$ are the components of a $B_{n}$-generalized pseudo-K\"{a}hler structure on the Courant algebroid $E_{\tilde{H}}: =E_{\tilde{H}, 0}$ where $\tilde{H} := \epsilon H.$ 
\end{cor}

\begin{rem} \label{2-even}  The formulation and the proof of Theorem \ref{1-even} can be adapted to the case when the vector fields $X_{\pm }$ 
are trivial (see Remark \ref{classical-even}). Consider the setting of  Theorem~\ref{1-even} but assume that 
$X_{+} = X_{-} =0.$  It turns out that  $(\mathcal{G}, \mathcal F)$ is $B_{n}$-generalized pseudo-K\"{a}hler  if and only if $F =0$ and
$(\mathcal{G}, \mathcal F )$ is a generalized pseudo-K\"{a}hler structure on the even exact Courant algebroid 
$TM \oplus T^{*}M$  (trivially extended to $TM\oplus T^*M \oplus \mathbb{R}$) with Dorfman  bracket twisted by the (closed) $3$-form $H$. Such structures were  defined 
in \cite{gualtieri-thesis} and  are deeply studied in the literature (especially  for positive definite signature).
\end{rem}

\subsection{Proofs of Theorems \ref{integrability-odd} and \ref{1-even} }\label{proofs-section}

We start with various general integrability results, which will be used in the proofs of Theorems \ref{integrability-odd} and \ref{1-even}.

\subsubsection{General integrablity results}

Let $(\mathcal{G}, \mathcal F )$ be a $B_{n}$-generalized almost pseudo-Hermitian structure on an odd exact  Courant algebroid $E$
over a manifold $M$ of dimension $n$. Let 
$L_{1}\subset E_{\mathbb{C} }$  be the $i$-eigenbundle of $\mathcal F$
and $E_{\pm}$ the $\pm 1$-eigenbundles of $\mathcal{G}^{\mathrm{end}}.$  Since $\mathcal{G}^{\mathrm{end}}$ commutes with $\mathcal F$, 
$$
L_{1} = L _{1} \cap  (E_{+} )_{\mathbb{C}}\oplus L_{1} \cap (E_{-} )_{\mathbb{C}}= L_{1} \cap L_{2} \oplus L_{1} \cap \bar{L}_{2}
$$
where $L_{2}$ is the  $i$-eigenbundle of $\mathcal F_{2} := \mathcal F \mathcal{G}^{\mathrm{end}}.$ 
Let $L_{1}^{+} : =L_{1} \cap L_{2}$ and $L_{1}^{-} :=L_{1} \cap \bar{L}_{2}$ and  
$u_{0}\in \mathrm{Ker}\, \mathcal F$ such that  $\langle u_{0}, u_{0}\rangle = (-1)^{n}$.

\begin{lem}\label{crit-first} Assume that $\mathcal F$ is integrable. The  Dorfman Lie derivative  ${\mathbf L}_{u_{0}}:\Gamma (E) \rightarrow \Gamma (E)$ preserves  $\Gamma (L_{1}^{+})$  if and only if
it preserves $\Gamma (L_{1}^{-})$.
\end{lem}

\begin{proof}
We remark that 
\begin{align}
\nonumber& L_{1}^{-}  = \{ v\in L_{1}\mid  \langle v, w\rangle =0,\ \forall w\in \overline{ L_{1}^{+} } \} \\
\label{charact-dec}& L_{1}^{+} =\{ v\in L_{1}\mid \langle v, w\rangle =0,\ \forall w\in \overline{ L_{1}^{-}}\} .
\end{align}
Assume  that ${\mathbf L}_{u_{0} }$ preserves 
$\Gamma ( L_{1}^{+})$.  Then ${\mathbf L}_{u_{0} }$ preserves  also
$\Gamma ( \overline{L_{1}^{+}})$. 
For any 
$v\in \Gamma (L_{1}^{-})$ and  $w\in \Gamma ( \overline{L_{1}^{+}})$,  
\begin{equation}\label{dorf-crit}
\langle   [u_{0} , v ] , w\rangle  = - \langle  v, [u_{0} , w ] \rangle =  0,
\end{equation}
where we used $\langle v, w\rangle =0$, 
$[u_{0}, w]\in \Gamma (\overline{L_{1}^{+}})$ and the first relation 
(\ref{charact-dec}).  
Since $\mathcal F$ is integrable,  $[u_{0}, v]\in \Gamma (L_{1})$
(see Lemma \ref{dorf-u0}). 
We conclude  from   the first relation  (\ref{charact-dec})  and (\ref{dorf-crit}).
The converse statement follows in a similar way, by using the second relation
(\ref{charact-dec}).
\end{proof}

The next proposition  is the analogue of Proposition 6.10 of \cite{gualtieri-thesis}.

\begin{prop}\label{criterion}
Let $(\mathcal{G}, \mathcal F)$  be a $B_{n}$-generalized almost pseudo-Hermitian structure. 
Then 
 $(\mathcal{G}, \mathcal F)$ is integrable if and only if  
$L_{1}$, $L_{1}^{\pm}$ are integrable and any of the equivalent conditions from
Lemma  \ref{crit-first} hold.
\end{prop}

\begin{proof}  If $(\mathcal{G}, \mathcal F )$ is integrable  then obviously 
$L_{1}^{\pm } $ 
and $L_{1}$ are integrable. From  
Lemma 4.13 of \cite{rubio} , 
${\mathbf L}_{u_{0}}$ preserves 
$\Gamma (L_{1}^{\pm })$.

For the converse,  we need to show that 
$L_{2} = L_{1}^{+} \oplus  \overline{L_{1}^{-}}$ is  integrable, i.e.
\begin{equation}\label{lie-plus-minus}
[ \Gamma ( L_{1}^{+}), \Gamma   (\overline{L_{1}^{-}} ) ] \subset \Gamma
( L_{1}^{+} + \overline{L_{1}^{-}} ).
\end{equation}
The map which assigns to $X\in \bar{L}_{1}$ the covector $\xi \in L_{1}^{*}$ defined by $\xi ( Y)  := \langle X, Y\rangle$, for any
$Y\in L_{1}$,  is an isomorphism, which maps 
$\overline{L_{1}^{\pm }}\subset \bar{L}_{1}$ onto $\mathrm{Ann}\, (L_{1}^{\mp}) \subset L_{1}^{*}$.
We identify $\bar{L}_{1}$ with $L_{1}^{*}$ by means of this isomorphism. In particular,  
${L}^{*}_{1}$ inherits a Lie algebroid structure from the Lie algebroid
structure of  $\bar{L}_{1}$.  
As proved in \cite{rubio} (Section 4.4, page 66), 
for any $X\in \Gamma (L_{1})$ and $\xi \in \Gamma ( L_{1}^{*})$, 
\begin{equation}\label{lie-dual}
[ X, \xi ] = {\mathcal L}_{X} \xi - i_{\xi} d_{L_{1}^{*}} X + (-1)^{n} \langle [ u_{0}, X], \xi \rangle u_{0}
\end{equation}
where 
$\mathcal L_{X}$ denotes the Lie derivative of  the Lie algebroid $L_{1}$, defined by the Cartan formula 
${\mathcal  L}_{X} := i_{X} d_{L_{1} }+ d_{L_{1}}  i_{X}$,  and 
\begin{align}
\nonumber& d_{L_{1}}: \Gamma ( \Lambda^{k} L_{1}^{*}) \rightarrow \Gamma ( \Lambda^{k+1} L_{1}^{*}),
\nonumber& d_{L_{1}^{*}}: \Gamma ( \Lambda^{k} L_{1}) \rightarrow \Gamma ( \Lambda^{k+1} L_{1}),
\end{align}
are the exterior derivatives of the Lie algebroids  $L_{1}$ and $L_{1}^{*}.$  
Assume now that  $X\in \Gamma (L_{1}^{+})$ and $\xi
\in\Gamma (\overline{ L_{1}^{-}}).$  
Since 
${\mathbf L}_{u_{0}}$ preserves $\Gamma (L_{1}^{+})$, the last term  from (\ref{lie-dual})  vanishes and we obtain
\begin{equation}\label{lie-dual2}
[ X, \xi ] = {\mathcal L}_{X} \xi - i_{\xi} d_{L_{1}^{*}} X =  i_{X} d_{L_{1}} \xi  - i_{\xi} d_{L_{1}^{*}} X,
\end{equation}
where we used $\xi (X) =0.$
On the other hand, if  $L$ is an arbitrary Lie algebroid with a decomposition $L = L_{1} \oplus L_{2}$ where
$L_{i}$ are integrable subbundles of $L$, then $ i_{X} d_{L} \xi \in \Gamma (  \mathrm{Ann}\, (L_{1}))$, for any
$X\in \Gamma (L_{1})$ and $\xi \in \Gamma ( \mathrm{Ann}\,  (L_{1})).$  
Applying this result to the Lie algebroids 
$L_{1} = L_{1}^{+} \oplus L_{1}^{-}$ and 
$$
L_{1}^{*} =  (L_{1}^{+} )^{*} \oplus  (L_{1}^{-})^{*} = \mathrm{Ann}\,  (L_{1}^{-}) \oplus  \mathrm{Ann}\, (L_{1}^{+}) =
\overline{L_{1}^{+}} \oplus \overline{L_{1}^{-}}
$$
we obtain that $i_{X} d_{L} \xi $ is a section of $\mathrm{Ann}\, ( L_{1}^{+}) = \overline{L_{1}^{-}}$ 
and 
$i_{\xi} d_{L_{1}^{*}} X $ is a section of  $\mathrm{Ann}\, (\mathrm{Ann}\, ( L_{1}^{+}) )= L_{1}^{+}$.
Relation (\ref{lie-plus-minus})  follows.
\end{proof}

Let $E = E_{H, F}$ be a Courant algebroid  of type $B_{n}$ over a manifold $M$ with Dorfman bracket 
twisted  by $(H, F)$. On $M$ we consider a  pseudo-Riemannian metric $h$,  
a complex distribution 
$\mathcal D\subset (TM)_{\mathbb{C}}$ and a vector field 
$X_{0}$ (real or complex), $h$-orthogonal
to $\mathcal D$ and 
such that $h(X_{0}, X_{0})=1.$  We define  the subbundles
\begin{align}
\nonumber& L^{h, \mathcal D} := \{  X + h(X) \mid X\in \mathcal D\},\\
\nonumber& L^{h, \mathcal D , X_{0}} :=  L^{h, \mathcal D}  \oplus\mathrm{span}_{\mathbb{C}}  \{ X_{0} + h(X_{0}) + i \} \subset E_\mathbb{C}. 
\end{align} 
Let $Y_{0}\in {\mathfrak X}(M)$,  $f\in C^{\infty}(M, \mathbb{C})$ and 
$$
u^{Y_{0}} :=  Y_{0} + h(Y_{0}),\ u^{Y_{0}, f}:= Y_{0} - h(Y_{0}) + f,
$$
which are sections of  $E_\mathbb{C}$.
Let $\nabla$ be the Levi-Civita connection of $h.$

\begin{lem}\label{general-results}  i) The bundle $L^{h, \mathcal D}$ is integrable if and only if $\mathcal D$ is involutive, $F\vert_{\Lambda^{2} \mathcal D} =0$
and, for any  $X, Y\in \Gamma (\mathcal D )$ and  $Z\in {\mathfrak X}(M)$, 
\begin{equation}\label{first-rel}
h (\nabla_{Z} X, Y) =\frac{1}{2} H(X, Y, Z).
\end{equation}
ii) The Dorfman Lie derivative  ${\mathbf L}_{u^{Y_{0}}}$ preserves $\Gamma (L^{h,\mathcal D})$ 
if and only if 
$$
{\mathcal L}_{Y_{0}}\Gamma (\mathcal D)
\subset \Gamma (\mathcal D),\ (i_{Y_{0}} F)\vert_{\mathcal D} =0
$$ 
and, for any $ X\in \Gamma (\mathcal D )$ and  $Z\in {\mathfrak X}(M)$, 
\begin{equation}
h (\nabla_{Z} Y_{0}, X) =\frac{1}{2} H(Y_{0}, X, Z).
\end{equation}

iii)  The Dorfman Lie derivative  ${\mathbf L}_{u^{Y_{0}, f}}$ preserves $\Gamma (L^{h,\mathcal D})$ 
if and only if 
$$
{\mathcal L}_{Y_{0}} \Gamma (\mathcal D )\subset\Gamma (\mathcal D),\
(i_{Y_{0}} F)\vert_{ \mathcal D} = (df)\vert_{\mathcal D}
$$  and, for any  $X\in \Gamma (\mathcal D)$ and  $Z\in {\mathfrak X}(M)$,     
\begin{equation}
h(\nabla_{X} Y_{0}, Z) =\frac{1}{2}H(Y_{0}, X, Z)- f F(X, Z).
\end{equation}

iv) The bundle $L^{h, \mathcal D , X_{0}}$ is integrable if and only if  
the following conditions hold:

\begin{enumerate}

\item  for any $Y, Z\in \Gamma (\mathcal D )$,
\begin{align}
\nonumber& {\mathcal L}_{Y} Z + i F(Y, Z) X_{0}\in \Gamma 
(\mathcal D )\\
\label{lie} & {\mathcal L}_{X_{0}}Y+ i F(X_{0}, Y)X_{0} \in \Gamma (\mathcal D ) ;
\end{align}  
\item   for any $Y, Z\in \Gamma (\mathcal D )$ and $X\in {\mathfrak X}(M)$, 
\begin{align}
\nonumber& h (\nabla_{X} Y, Z) =\frac{1}{2} H(X, Y, Z)\\
\label{conn-even-1}&  h (\nabla_{X}  X_{0}, Y) =- \frac{1}{2} ( i_{X_{0}} H) (X, Y) +  i  F(X, Y) .
\end{align}
\end{enumerate}

\end{lem}

\begin{proof} The proof is a straightforward computation,  which uses the expression~(\ref{dorfmann}) of the Dorfman bracket. For example, to prove claim  i)  let $X, Y\in \Gamma (\mathcal D )$. Then  
\begin{equation}
[ X + h(X), Y + h(Y) ] = {\mathcal L}_{X}Y + {\mathcal L}_{X} \left( h(Y)\right) -  i_{Y} d h(X) - H(X, Y, \cdot )
+ F(X, Y)
\end{equation}
is a section of $L^{h,\mathcal D}$ if and only if ${\mathcal L}_{X}Y \in \Gamma (\mathcal D)$,  
$F(X, Y)=0$ 
and
$$
h( \mathcal L_{X}Y) =  {\mathcal L}_{X} (h(Y))- i_{Y} d h(X) -  H(X, Y,\cdot ).
$$ 
Applying the above relation  to $Z\in {\mathfrak X}(M)$ and writing it in terms of $\nabla$ 
we obtain (\ref{first-rel}).
The other claims can be obtained similarly.  (For the last equation we observe that 
$h (\nabla_{X}  X_{0}, Y)=-h(\nabla_XY,X_0)$, since $X_0\perp Y$.)

\end{proof}

\begin{lem}\label{general-results-1}  
Let $h$ be a pseudo-Riemannian metric on $M$,   $\mathcal D_{\pm}\subset (TM)_{\mathbb{C}}$  two complex  distributions 
such that $\mathcal D_{+}$ is isotropic
and $X_{0}\in  {\mathfrak X}(M)_\mathbb{C}$, such that $X_{0}$ is
orthogonal to $\mathcal D_{+}$ and   $h ( X_{0}, X_{0}) =1.$ 
Assume that  the following relations hold:\

i) for any $X\in \Gamma (\mathcal D_{-})$ and $Y\in \Gamma (\mathcal D_{+ })$,
\begin{equation}\label{n1}
\nabla_{Y}X -\frac{1}{2} H(X, Y) \in \Gamma (\mathcal D_{-});
\end{equation}
ii) for any $X\in \Gamma (\mathcal D_{-})$ and $Y\in \Gamma (\mathcal D_{+})$, 
\begin{equation}\label{n2}
\nabla_{X} Y -\frac{1}{2} H(X, Y) \in \Gamma  (\mathcal D_{+}\oplus \mathrm{span}_{\mathbb{C}}\{ X_{0}\} ));
\end{equation}
iii) for any $X\in \Gamma (\mathcal D_{-})$,
\begin{equation}\label{n3}
\nabla_{X_{0}} X + \frac{1}{2} H(X_{0}, X) \in \Gamma (\mathcal D_{-} );
\end{equation}
iv) for any $X\in \Gamma (\mathcal D_{-})$,   
\begin{equation}\label{n4}
\nabla_{X}X_{0} + \frac{1}{2} H(X_{0}, X)- iF(X) \in \Gamma (\mathcal D_{+} \oplus \mathrm{span}_{\mathbb{C}}\{ X_{0}\} ); 
\end{equation}
v) for $X\in \Gamma (\mathcal D_{-})$, 
\begin{equation}\label{f-d}
F(X) \in \Gamma (\mathcal D_{-}).
\end{equation}
Then 
\begin{equation}\label{partial-integrable}
[ \Gamma ( L^{-h, {\mathcal D}_{-}} ), \Gamma ( L^{h, \mathcal D_{+}, X_{0} } ) ] \subset 
\Gamma ( L^{-h,\mathcal D_{-}} \oplus  L^{h, \mathcal D_{+}, X_{0}}).
\end{equation}
\end{lem}

\begin{proof}
In order to prove (\ref{partial-integrable}) we need to show that   $[ X - h(X), Y + h(Y) ] $
and  $[ X - h(X), X_{0} + h(X_{0}) + i ] $ are sections of 
$L^{-h,\mathcal D_{-}} \oplus  L^{h, \mathcal D_{+}, X_{0}}$, for any  
$X\in \Gamma (\mathcal D_{-})$ and $Y\in   \Gamma (\mathcal D_{+})$.
Remark that 
\begin{align*}
\nonumber&  L^{-h, {\mathcal D}_{-}} \oplus  L^{h, \mathcal D_{+}, X_{0} }= \\
\nonumber&\{ X - h(X)\mid X\in \mathcal D_{-} \} + \{ Y+ h(Y) + i h(Y, X_{0})\mid Y\in \mathcal D_{+} \oplus  \mathrm{span}_{\mathbb{C}} \{ X_{0}\} \} .
\end{align*}
Writing $\mathcal L_{X} Y = {\mathcal F}^{\prime}_{+}(X, Y) - {\mathcal F}^{\prime}_{-}(X, Y) $, 
where 
\begin{align*}
\nonumber& {\mathcal F}^{\prime}_{+}(X, Y):=  \nabla_{X} Y- \frac{1}{2} H (X, Y)\in \Gamma (\mathcal D_{+} \oplus \mathrm{span}_{\mathbb{C}}\{ X_{0}\} ),\\
\nonumber& {\mathcal F}^{\prime}_{-}(X, Y):= \nabla_{Y} X-\frac{1}{2} H(X, Y)\in \Gamma (\mathcal D_{-}), 
\end{align*} 
(compare relations (\ref{n1}) and (\ref{n2})), we obtain 
\begin{align*}
\nonumber & [ X - h(X), Y + h(Y) ] \\
\nonumber&  ={\mathcal F}^{\prime}_{+}(X, Y) - {\mathcal F}^{\prime}_{-}(X, Y)  + \mathcal L_{X} \left( h(Y)\right)  + i_{Y} d h(X) - H(X, Y, \cdot )
+ F(X, Y).
\end{align*}
Relation (\ref{n4}), together with the fact that  $\mathcal D_{+}$ is isotropic and $X_{0}$ is orthogonal to $\mathcal D_{+}$,  
imply that 
\begin{align}
\nonumber&  F(X, Y) = i h({\mathcal F}^{\prime}_{+}(X, Y) , X_{0}),\\
 \nonumber& \mathcal L_{X} (h(Y) )+ i_{Y} d h(X) - H(X, Y,\cdot ) = h( {\mathcal F}^{\prime}_{+}(X, Y)) + h ( {\mathcal F}^{\prime}_{-}(X, Y) ),
 \end{align}
from where we deduce that 
\begin{align}
\nonumber [ X - h(X), Y + h(Y) ] &=  {\mathcal F}^{\prime}_{+}(X, Y) +h ( {\mathcal F}^{\prime}_{+}(X, Y) )+ 
i h ( {\mathcal F}^{\prime}_{+}(X, Y), X_{0})\\ 
\label{partial-integrable-2}&  - {\mathcal F}^{\prime}_{-}(X, Y) + h( {\mathcal F}^{\prime}_{-}(X, Y))
\end{align}
is a section of  $L^{-h,\mathcal D_{-}} \oplus  L^{h, \mathcal D_{+}, X_{0}}.$
Similar computations show that 
\begin{align}
\nonumber [ X - h(X), X_{0} + h(X_{0}) + i ] &=  {\mathcal F}_{+} (X) + h ({\mathcal F}_{+} (X))  + i  h ({\mathcal F}_{+} (X), X_{0} )\\
\label{partial-integrable-1}& - {\mathcal F}_{-} (X) +   h ({\mathcal F}_{-} (X)) 
\end{align} 
is also a section of $L^{-h,\mathcal D_{-}} \oplus  L^{h, \mathcal D_{+}, X_{0}}$, where
\begin{align*}
\nonumber& \mathcal F_{+} (X) := \nabla_{X} X_{0} +\frac{1}{2} H(X_{0}, X) - i F(X)  \in \Gamma (\mathcal D_{+} \oplus \mathrm{span}_{\mathbb{C}} \{ X_{0} \} ),\\
\nonumber& \mathcal F_{-} (X) := \nabla_{X_{0}} X+\frac{1}{2} H(X_{0}, X) - i F(X)\in \Gamma (\mathcal D_{-}),
\end{align*} 
(from relations (\ref{n3}), (\ref{n4})  and
(\ref{f-d})).  
\end{proof}

\subsubsection{Application of  general integrability results }

We now turn to the setting of Section \ref{integrability-statements}, and, using the results
from the previous section, we prove 
Theorems \ref{integrability-odd} and \ref{1-even}.
Consider the setting of these theorems. 
We start by computing  the bundles $L_{1}^{\pm}$ associated to  the $B_{n}$-generalized pseudo-Hermitian structure 
$(\mathcal{G}, \mathcal F )$, as in Proposition \ref{criterion}. 
They  turn out to be of the form $L^{h,\mathcal D}$ or $L^{h, \mathcal D, X_{0}}$ for suitably chosen $h$,
$\mathcal D$ and $X_{0}.$

\begin{lem}\label{computation-l-pm}  i) If $ n$ is odd then  
\begin{align}
\nonumber L_{1}^{+}& = \{   X + g(X) \mid   X\in T_{J_{+}}^{(1,0)}M\} \oplus 
\mathrm{span}_{\mathbb{C}}\{ X_{+} + g(X_{+}) + i\};\\ 
\label{cpm-odd}L_{1}^{-} &= \{ X - g(X)\mid X\in T_{J_{-}}^{(1,0)}M\} .
\end{align}
ii) If $ n$ is even 
then
\begin{align}
\nonumber L_{1}^{+} &= \{   X  + g(X)\mid X\in   T^{(1,0)}_{J_{+}}M \}\oplus \mathrm{span}_{\mathbb{C}} \{ V_{-} + g (V_{-}) + i \} ;\\
\label{int-1} \L_{1}^{-} &= \{  X - g(X)\mid X\in T_{J_{-}}^{(1,0)} M\} ,
\end{align}
where $V_{-} :=\frac{1}{ 1- c_{+}^{2} } ( X_{-} - i c_{+} X_{+})$
is of norm one and orthogonal to $T_{J_{+}}^{(1,0)}M$.
\end{lem}

\begin{proof}
The proof is  straightforward from Proposition
\ref{data-on-M}. In particular, one can easily check from the formulas for 
$\mathcal{J}_+$ and $\mathcal{F}_+$ in Proposition~\ref{data-on-M} that 
$\mathcal{J}_+(X_++i) = i(X_++i)$ and $\mathcal{F}_+(V_-+i)=i(V_-+i)$, which are equivalent to 
$\mathcal{F}(X_+ + g(X_+) + i) = i(X_+ + g(X_+) + i)$ and $\mathcal{F}(V_-+g(V_-)+i)=i(V_-+g(V_-)+i)$, respectively.
\end{proof}

\begin{cor}\label{pre-final-1}  Assume that $n$ is odd. 
The following are equivalent:\

i) $( \mathcal{G}, \mathcal F )$ is a $B_{n}$-generalized pseudo-K\"{a}hler structure;\

ii) $L_{1}^{\pm}$ are integrable  and   the Dorfman Lie derivative ${\mathbf L}_{u_{-}}$
preserves $\Gamma ( L_{1}^{-})$, where $u_{-}:= X_{-} - g(X_{-})$;\

iii)   $T_{J_{-}}^{(1,0)}M$ is involutive,  $F\vert_{\Lambda^{2} T_{J_{-}}^{(1,0)}M}=0$,
$i_{X_{-}} F=0$, the Lie derivative ${\mathcal L}_{X_{-}}$ preserves $\Gamma (T_{J_{-}}^{(1,0)}M)$
and the following relations hold:

\begin{enumerate}

\item for any $X, Y\in \Gamma (T_{J_{-}}^{(1,0)}M)$ and $Z, V\in {\mathfrak X}(M)$, 
\begin{align}
\nonumber& g (\nabla_{Z} X, Y) =-\frac{1}{2} H(X, Y, Z),\\
\label{nabla}&  g(\nabla_{Z} X_{-}, V) = -\frac{1}{2} H(X_{-}, V, Z);
\end{align}

\item for any $Y, Z\in \Gamma (T_{J_{+}}^{(1,0)}M)$, 
\begin{align}
\nonumber& {\mathcal L}_{Y} Z + i F(Y, Z) X_{+}\in \Gamma (T_{J_{+}}^{(1,0)}M)\\
\label{lie-conditii}& {\mathcal L}_{X_{+}} Z + i F(X_{+}, Z) X_{+}\in \Gamma (T_{J_{+}}^{(1,0)}M);
\end{align}

\item for any $Y, Z\in \Gamma (T_{J_{+}}^{(1,0)}M)$ and $X\in {\mathfrak X}(M)$,
\begin{align}
\nonumber& g (\nabla_{X} Y, Z) =\frac{1}{2} H(X, Y, Z)\\
\label{nabla-plus} & g (\nabla_{X} X_{+}, Y) =- \frac{1}{2}( i_{X_{+}} H)(X, Y) + i  F(X, Y).
\end{align}
\end{enumerate}
\end{cor}

\begin{proof} The implication i) $\Longrightarrow $ ii) follows from
Proposition \ref{criterion}, while the equivalence between ii) and iii) follows from 
Lemma \ref{general-results} i) with $h:= -g$ and  $\mathcal D:= {T}_{J_{-}}^{(1,0)}M$, 
Lemma \ref{general-results} iv) with $h:= g$, $\mathcal D:= {T}_{J_{+}}^{(1,0)}M$
and $X_{0} := X_{+}$, and Lemma \ref{general-results} ii) with $h:= - g$,
$\mathcal D:= T_{J_{-}}^{(1,0)}M$ and  $Y_{0} := X_{-}$. In order to prove 
that iii) implies i)
we apply again Proposition \ref{criterion}. 
We need to show 
that the relations from iii) imply that $L_{1}$ is integrable, or 
\begin{equation}
[ \Gamma (L_{1}^{+}), \Gamma ( L_{1}^{-} ) ] \subset  \Gamma (L_1) =\Gamma (L_{1}^{+}\oplus L_{1}^{-}).
\end{equation}
The above relation  follows  from Lemma \ref{general-results-1}, by  
noticing that the assumptions
from this lemma   
are implied by the conditions 
from iii).  (The lemma is specialized to $\mathcal{D}_\pm = T^{(1,0)}_{J_\pm}M$, $h=g$ and $X_0=X_+$).
For example, the first relation
(\ref{nabla}) implies that 
$\nabla_{Y} X - \frac{1}{2} H(X, Y)$ is orthogonal to  ${T}_{J_-}^{(1,0)}M$, i.e.\ is  a section of
 ${T}_{J_{-}}^{(1,0)}M\oplus \mathrm{span}_{\mathbb{C}}\{ X_{-} \} $, 
for any  $X\in \Gamma ({T}_{J_{-}}^{(1,0)}M)$ and  $Y\in \Gamma (T_{J_{+}}^{(1,0)}M)$.
From the second relation (\ref{nabla}),  it is also orthogonal to  $X_{-}$.
Relation (\ref{n1}) follows. The other relations can be checked similarly.
\end{proof}

\begin{cor}\label{pre-final-2} Assume that $n$ is even. The following are equivalent:

i) $(\mathcal{G}, \mathcal F  )$ is a $B_{n}$-generalized  pseudo-K\"{a}hler structure;\

ii)  $L_{1}^{\pm}$ are integrable  and   the Dorfman Lie derivative ${\mathbf L}_{u_{+}}$
preserves $\Gamma ( L_{1}^{-})$, where $u_{+}:= X_{+} + g(X_{+}) +c_{+}$;\

iii)  $J_{-}$ is integrable,  $F\vert_{\Lambda^{2}T^{(1,0)}_{J_{-}}M}=0$  and the following relations hold:

\begin{enumerate}

\item for any  $X, Y\in\Gamma ( { T}^{(1,0)}_{J_{-}}M)$ and 
$Z\in {\mathfrak X}(M)$, 
\begin{align}
\nonumber& g( \nabla_{Z} X, Y) =-\frac{1}{2} H(X, Y, Z),\\ 
\label{even-B-u}& g(\nabla_{X} X_{+}, Z) =- \frac{1}{2}  ( i_{X_{+}} H)(X, Z) +  c_{+} F(X, Z) ;
\end{align}

 \item  for any $X, Y\in \Gamma (T_{J_{+}}^{(1,0)}M)$,
\begin{align}
\nonumber& {\mathcal L}_{X} Y + i F(X, Y) V_{-}\in \Gamma 
(T^{(1,0)}_{J_{+}}M)\\
\label{lie} & {\mathcal L}_{V_{-}}X+ i F(V_{-}, X)V_{-} \in \Gamma (T^{(1,0)}_{J_{+}}M),
\end{align}
where, we recall,  $V_{-} =\frac{1}{1- c_{+}^{2}}( X_{-} - i c_{+} X_{+}).$\ 
 
\item for any $X, Y\in \Gamma  (T_{J_{+}}^{(1,0)}M)$ and $Z\in {\mathfrak X}(M)$, 
\begin{align}
\nonumber& g(\nabla_{Z} X, Y) =\frac{1}{2} H(X, Y, Z)\\
\label{conn-even}& g (\nabla_{Z}  V_{-}, X) =\frac{1}{2} ( i_{V_{-}} H) (X, Z) - i  F(X, Z),  
\end{align} 

\item \begin{equation}\label{lasteq:eq}i_{X_{+}}F = dc_+.\end{equation}

\end{enumerate}
\end{cor}

\begin{proof} The claims follow with a similar argument, as in Corollary \ref{pre-final-1}.
We only explain the implication ii) $\Longrightarrow$ iii). 
We use
Lemma \ref{general-results} i) with $h:= -g$
and $\mathcal D:= T_{J_{-}}^{(1,0)}M$,  Lemma \ref{general-results} iv) with $h:= g$,
$\mathcal D :=T_{J_{+}}^{(1,0)}M$,  $X_{0} :=V_{-}$,  and Lemma  \ref{general-results} iii) 
with $h:= - g$, $\mathcal D := T_{J_{-}}^{(1,0)}M$, 
$Y_{0}:= X_{+}$, $f:= c_{+}$. We obtain  that $L_{1}^{\pm}$ are integrable and  
the Dorfman Lie derivative  ${\mathbf L}_{u_{+}}$
preserves $\Gamma ( L_{1}^{-})$ if and only if the relations  1., 2.\ and 3.\ from iii) hold, together with 
$i_{X_{+}} F = dc_{+}$ on $T_{J_{-}}^{(1,0)} M$ and ${\mathcal L}_{X_{+}} \Gamma ( T_{J_{-}}^{(1,0)}M) \subset
\Gamma ( T_{J_{-}}^{(1,0)} M)$.   However, the real one-form $i_{X_{+}} F -dc_{+}$ vanishes on 
$T_{J_{-}}^{(1,0)} M$ if and only if it is zero, since $J_-$ is an almost complex structure. Finally, ${\mathcal L}_{X_{+}} \Gamma ( T_{J_{-}}^{(1,0)}M) \subset
\Gamma ( T_{J_{-}}^{(1,0)} M)$
follows   
by combining the first relation (\ref{even-B-u})  with $Z:= X_{+}$
with   the second relation (\ref{even-B-u}) with $Z:= Y$ and using that $F(X, Y)=0$  
for any $X, Y\in \Gamma (T^{(1,0)}_{J_{-}} M)$,
since $F$ is of type 
$(1,1)$ with respect to $J_{-}$.
\end{proof}

\begin{rem} We remark that on the support of the function $c_+$ the relation 
(\ref{lasteq:eq}) 
is automatically satisfied if the 
other conditions listed under iii) in Corollary~\ref{pre-final-2}  hold. 
In fact, 
\[ (i_{X_{+}} F -dc_+)|_{T_{J_{-}}^{(1,0)}M}=0,\] 
and hence (\ref{lasteq:eq}),  
follows from  the second relation
(\ref{even-B-u}), by letting $Z:= X_{+}$ and using $g(X_{+}, X_{+}) = 1-c_{+}^{2}.$ 
\end{rem}

We conclude the proofs of Theorems \ref{integrability-odd} and  \ref{1-even} by noticing
that the conditions from
Corollary \ref{pre-final-1}  iii) and \ref{pre-final-2} iii) 
are equivalent to the conditions from these theorems.
This is done in the next two lemmas.

\begin{lem}
Assume that $(g,  J_{+}, J_{-}, X_{+}, X_{-})$ 
are the components of a $B_{n}$-generalized almost pseudo-Hermitian structure on
a  Courant algebroid $E = E_{H, F}$ of type $B_{n}$ over an odd dimensional manifold $M$.
The conditions from Corollary \ref{pre-final-1} iii) are equivalent to the conditions from Theorem \ref{integrability-odd}. 
\end{lem}

\begin{proof}
Assume that the conditions from Corollary \ref{pre-final-1} iii) are satisfied. 
Relations (\ref{nabla}) imply that $\nabla^{-}$ preserves 
$T^{(1,0)}_{J_{-}} M$  and $\nabla^{-} X_{-}=0$.  
The second relation  (\ref{nabla-plus}) implies that
\begin{equation}
\nabla_{X} X_{+} +\frac{1}{2} H(X_{+}, X) - i F(X) \in \Gamma ( T^{(1,0)}_{J_{+}} M +\mathrm{span}_{\mathbb{C}}\{ X_{+} \} ) ,
\end{equation}
for any $X\in {\mathfrak X}(M)$.
Using that $g(X_{+}, X_{+}) =1$ we obtain 
\begin{equation}\label{x-plus}
\nabla_{X} X_{+} +\frac{1}{2} H(X_{+}, X) - i  ( F(X)  + F(X_{+}, X) X_{+}) \in \Gamma ( T^{(1,0)}_{J_{+}} M ).
\end{equation}
Since 
 both, the real and the imaginary parts of the left-hand side of (\ref{x-plus}) 
are orthogonal to  $X_{+}$,  $J_{+}\vert_{\mathrm{span}\{ X_{+} \}^{\perp}}$ is a complex structure and $J_{+} X_{+} =0$, 
we  deduce  from (\ref{x-plus}) that
\begin{equation}\label{X-plus}
\nabla_{X} X_{+} =-\frac{1}{2} H(X_{+}, X) - J_{+} F(X),\ \forall X\in {\mathfrak X}(M),
\end{equation}
which is equivalent to the second relation (\ref{covariant-deriv}). 
We now prove that $\nabla^{+}$ preserves $T^{(1,0)}_{J_{-}} M$. 
The first relation (\ref{nabla-plus}) implies that for any $X\in \mathfrak{X}(M)$ and 
$Y\in \Gamma  (T^{(1,0)}_{J_{+}}M)$, 
\begin{equation}\label{XX-plus}
\nabla_{X} Y -\frac{1}{2} H(X, Y) \in \Gamma (T^{(1,0)}_{J_{+}}M +\mathrm{span}_{\mathbb{C}} \{ X_{+} \} ).
\end{equation}
It is easy to see that (\ref{X-plus}) combined with (\ref{XX-plus}) 
imply that  $\nabla^{+}$ preserves  $T^{(1,0)}_{J_{+}}M$. 
The relations from Theorem  \ref{integrability-odd} i) follow.

The relations from Theorem \ref{integrability-odd} ii) can be obtained by similar computations. For example,  to prove that 
$H\vert_{\Lambda^{3} T_{J_{-}}^{(1,0)}M}=0$  let $ X, Y\in \Gamma ( T^{(1,0)}_{J_{-}} M)$ and write  
\begin{equation}
\mathcal L_{X} Y = \nabla^{-}_{X}Y - \nabla^{-}_{Y} X - H(X, Y).
\end{equation}
Since $T^{(1,0)}_{J_{-}}M$ is involutive and 
 $\nabla^{-}$ preserves
$T^{(1,0)}_{J_{-}} M$,  we obtain that $H(X, Y)$ is 
a section of $T^{(1,0)}_{J_{-}} M$. In particular, $H\vert_{\Lambda^{3} T_{J_{-}}^{(1,0)}M}=0$. 
To prove that $H\vert_{\Lambda^{3} T_{J_{+}}^{(1,0)}M}=0$  
and 
$(i_{X_{+}}H)= i F$ on ${\Lambda^{2}T_{J_{+}}^{(1,0)}M}$, 
we use  the first equation (\ref{lie-conditii}) and the fact that the distribution $T^{(1,0)}_{J_{+}} M$
is invariant under the connection $\nabla^+$. 

In the same vein one proves that conditions from Theorem \ref{integrability-odd} imply those from Corollary \ref{pre-final-1} iii).
\end{proof}

It remains to explain how   Theorem \ref{1-even} follows from Corollary \ref{pre-final-2}.

\begin{lem}
Assume that $(g,  J_{+}, J_{-}, X_{+}, X_{-}, c_{+})$ 
are the components of a $B_{n}$-generalized almost pseudo-Hermitian structure on 
a  Courant algebroid $E = E_{H, F}$ of type $B_{n}$ over an even dimensional manifold $M$, such that
$X_{+}$ and $X_{-}$ are non-null. 
Then the conditions from Corollary \ref{pre-final-2} iii) are equivalent to the conditions from Theorem \ref{1-even}. 
\end{lem}

\begin{proof} 
Straightforward computations as in the previous argument  show that all conditions from Corollary \ref{pre-final-2} iii),
except the second relation (\ref{lie}), are equivalent to all conditions from 
Theorem \ref{1-even}, except 
$\mathcal L_{X_{+}} X_{-} =0$ (or 
relation (\ref{H-F-even-prime}), see Lemma \ref{exchange}). 
Assuming that  these equivalent conditions hold,
we show that  the second relation (\ref{lie}) 
is equivalent to relation (\ref{H-F-even-prime}).
Let $X\in \Gamma ( T^{(1,0)}_{J_{+}} M)$
and recall that $V_{-} =\frac{1}{ 1- c_{+}^{2}} (X_{-} - i c_{+} X_{+})$.  
Writing
\begin{align}
\nonumber& {\mathcal L}_{V_{-}} X= D^{+}_{V_{-}} X +\frac{1}{2} H(V_{-}, X) 
-\frac{ i F(V_{-}, X)}{ 1- c_{+}^{2}} ( X_{-} - i c_{+} X_{+})\\
\nonumber& - X ( \frac{1}{ 1 - c_{+}^{2}} ) ( X_{-} - i c_{+} X_{+}) -\frac{1}{ 1- c_{+}^{2}} 
( \nabla_{X} X_{-} - i X(c_{+}) X_{+} - i c_{+} \nabla_{X} X_{+})
\end{align} 
and using that $D^{+}$ preserves  $T^{(1,0)}_{J_{+}} M$, we obtain that the second relation (\ref{lie}) is equivalent to
the statement that the expression
\begin{align}
\nonumber&  \mathrm{Expr} (X):= \\
\nonumber& i F(V_{-}, X) V_{-}  +\frac{1}{2} H(V_{-}, X) -\frac{ i F(V_{-}, X)}{ 1- c_{+}^{2}} ( X_{-} - i c_{+} X_{+})\\
\nonumber& - X ( \frac{1}{ 1 - c_{+}^{2}} ) ( X_{-} - i c_{+} X_{+}) -\frac{1}{ 1- c_{+}^{2}} 
( \nabla_{X} X_{-} - i X(c_{+}) X_{+} - i c_{+} \nabla_{X} X_{+})
\end{align}
is a section of $T^{(1,0)}_{J_{+}} M$, for any
$X\in \Gamma (T^{(1,0}_{J_{+}} M).$ 
Equivalently, 
$\mathrm{Expr}(X)$  is orthogonal to 
$T^{(1,0)}_{J_{+}} M$,  $X_{+}$ and $X_{-}$. The condition  
$$
g ( \mathrm{Expr} (X),  Y) =0,\  \forall X, Y\in \Gamma  (T^{(1,0)}_{J_{+}} M)
$$ 
is equivalent to 
$(i_{V_{-}} H)\vert_{\Lambda^{2} T^{(1,0)}_{J_{+}} M}= i F\vert_{\Lambda^{2} T^{(1,0)}_{J_{+}}M}$, which follows from the third  relation (\ref{H-F-alg-even}) and the first relation
(\ref{H-F-even}). 
Similary,  one can show that  the relations
$$
g (\mathrm{Expr} (X), X_{\pm})=0,\ \forall X\in \Gamma (T^{(1,0)}_{J_{+}} M)
$$  are  equivalent to 
\begin{equation}\label{rel-h-f-c}
H(X_{+}, X_{-}, X) = c_{+}  F(X_{-}, X) - (J_{+} X) (c_{+}),
\end{equation}
for all  $X\in \mathrm{span}\{ X_{+}, X_{-}\}^{\perp} .$  
Using $dc_{+} = i_{X_{+}} F$ we obtain that  relation  (\ref{rel-h-f-c}) is equivalent to 
relation (\ref{H-F-even-prime}),  as needed.
\end{proof}

 We end this section with a comment which shows that  the notion of $B_{n}$-generalized K\"{a}hler structure includes the notion of generalized K\"{a}hler structure on exact Courant algebroids, as defined in \cite{gualtieri-thesis}.

\begin{rem}\label{classical-even}{\rm  Let $(\mathcal{G}, \mathcal F )$ be a $B_{n}$-generalized almost pseudo-Hermitian 
structure with components $(g,  J_{+}, J_{-}, X_{+}, X_{-}, c_{+})$  on a Courant algebroid $E_{H, F}$ of type $B_{n}$
over a manifold $M$ of even dimension. Assume 
that $X_{+} = X_{-}= 0$ (in particular $c_{+}^{2} = 1$). 
Then  $J_{\pm}$ are $g$-skew-symmetric 
almost complex structures   on $M$ and (in the notation introduced in Section \ref{components-hermitian})  $u_{-} = 1$. The bundles 
$L_{1}^{\pm}$ as defined in Proposition \ref{criterion} are given by
\begin{equation}\label{int-2}
L_{1}^{\pm } = \{   X  \pm g(X) \mid X\in  T^{(1,0)}_{J_{\pm}} M\}.
\end{equation}
Adapting the above  arguments one can show that 
$L_{1}^{\pm}$ are integrable 
and 
${\mathbf L}_{ u_{-}}$ preserves $ \Gamma (L_{1}^{\pm})$ 
if and only if
$F=0$, $J_{\pm}$ are (integrable) complex structures  and
$\nabla_{X} \pm \frac{1}{2} H(X)$ preserves $J_{\mp}.$ These conditions hold if and only if 
$(\mathcal{G}, \mathcal F )$ is a $B_{n}$-generalized
pseudo-K\"{a}hler  structure. In fact, if $\mathcal{G}$ is positive
definite, 
$(g, J_{+}, J_{-}, b:=0)$ is   the bi-Hermitian structure  of a generalized   K\"{a}hler
structure on the even exact Courant algebroid  $TM \oplus T^{*}M$ with Dorfman bracket twisted  by the closed $3$-form $H$ (see 
Theorem 6.28 of \cite{gualtieri-thesis}). }
\end{rem}

\section{$B_{2}$-, $B_{3}$-,  $B_{4}$-generalized  pseudo-K\"{a}hler structures}\label{3-4}

In this section we show that 
Theorems \ref{integrability-odd} and \ref{1-even} simplify considerably 
when $M$ has dimension two, three or four. Below we denote by $X^{\flat}$ the $1$-form 
$g$-dual to a vector field $X\in {\mathfrak X}(M).$

\begin{cor} \label{ex-2-dim}  i) Let $(g,  J_{+} J_{-}, X_{+}, X_{-}, c_{+})$ be the components of a $B_{2}$-generalized pseudo-K\"{a}hler 
structure  $(\mathcal{G}, \mathcal F)$ on a Courant algebroid $E_{F}:= E_{0,F}$ over a $2$-dimensional manifold $M$, such that $X_{\pm}$ are non-null.
There is $\epsilon_{0} \in \{ \pm 1\} $ such that 
\begin{equation}\label{ecuatii-2-dim}
J_{-} X_{+} = - \epsilon_{0} X_{-},\ J_{-} X_{-}  =\epsilon_{0} X_{+},\ J_{+} = \epsilon_{0} c_{+} J_{-}.
\end{equation} 
If $c_{+}(p) \neq 0$ for any $p\in M$ then $X_{+}$ is a Killing field  with $g(X_{+}, X_{+}) <1$ and $F=\frac{1}{2c_{+}} d X_{+}^{\flat}$. 
If $c_{+}( p) =0$ for any $p\in M$ then $X_{+}$ is a parallel unit field  and $F=0.$\

ii) Conversely, any pair $(g, X_{+})$  formed by a pseudo-Riemannian metric $g$ and a vector field $X_{+}$  
such that either $X_{+}$ is a Killing field with $g(X_{+}, X_{+}) <1$ or $X_{+}$ is a 
parallel unit field  defines a $B_{2}$-generalized pseudo-K\"{a}hler structure 
on $E_{F} = E_{0, F}$   with $F$ defined as in i). The vector field $X_{-}$ is arbitrarily chosen, orthogonal to $X_{+}$ and 
of the same norm as $X_{+}$, the endomorphisms $J_{\pm}$ are defined by (\ref{ecuatii-2-dim}) and
$c_{+} := \epsilon_{+} (1- g(X_{+} , X_{+}))^{1/2}$, where $\epsilon_{+} \in \{ \pm 1\} $.
\end{cor}

\begin{proof}
i) The first two relations (\ref{ecuatii-2-dim}) follow from the fact that 
$J_{-}$ is a $g$-skew-symmetric complex structure, 
$X_{\pm}$ are orthogonal of  the same norm and 
$M$ is $2$ dimensional. The last relation (\ref{ecuatii-2-dim}) follows from $J_{+} X_{+} =- c_{+} X_{-}$ and
$J_{+} X_{-} = c_{+} X_{+}.$ 
Since $M$ is $2$-dimensional,  $H=0$, $F$ 
is of type $(1,1)$ with respect to $J_{-}$ 
 and the conditions  from  Theorem \ref{1-even}  (see also Lemma \ref{exchange}) 
reduce to 
\begin{equation}\label{simpl-2-dim}
\nabla_{X} X_{+} = c_{+} F(X),\ \nabla_{X} X_{-} = - J_{+} F(X),\ \nabla_{X} J_{-} =0,\ i_{X_{+}} F = d c_{+}
\end{equation}
for any $X\in {\mathfrak X}(M)$. 
However,  using 
(\ref{ecuatii-2-dim}) one can show that the second and third relation (\ref{simpl-2-dim}) are implied by the first and fourth relation.
When $c_{+}(p) \neq 0$, for any $p\in M$,  the first relation (\ref{simpl-2-dim}) implies that $F = \frac{1}{ 2c_{+}} d X_{+}^{\flat}$
and we obtain that the last relation (\ref{simpl-2-dim}) is satisfied. When $c_{+} (p) =0$,  for any $p\in M$,  the last relation implies that
$F=0.$\

ii) Reversing the argument  from i) we obtain claim ii).  
\end{proof}

\begin{cor}\label{ex-3-dim} In the setting of Theorem \ref{integrability-odd} assume that  $M$ is $3$-dimensional. Then $(\mathcal{G}, \mathcal F )$ is  a  $B_{3}$-generalized pseudo-K\"{a}hler  structure if and only if
$i_{X_{-}} F =0$ and 
\begin{align}
\nonumber& \nabla_{X} X_{-} = -\frac{1}{2} H(X, X_{-})\\
 \label{gr-3}& \nabla_{X} X_{+} = -\frac{1}{2} H(X_{+}, X) - J_{+} F(X),
\end{align}
for any $X\in {\mathfrak X}(M).$ 
\end{cor}

\begin{proof} 
Relations (\ref{gr-3}) are just relations (\ref{covariant-deriv}),
written in terms of the Levi-Civita connection $\nabla$ of $g$.
Since  the bundles  $T_{J_{\pm}}^{(1,0)} M$ have  rank one, 
relations (\ref{rel-F-H}) reduce to $i_{X_{-}} F=0.$ 
It remains to prove that
$\nabla^{\pm}$ preserve $T^{(1,0)}_{J_{\pm}}M.$ Let $v_{-}$ be a non-zero local section of
$T^{(1,0)}_{J_{-}} M.$ Since $\nabla^{-}g=0$ and $v_{-}$ is isotropic, $g(\nabla^{-}_{X} v_{-}, v_{-})=0$.
Moreover,
$$
g(\nabla_{X}^{-} v_{-}, X_{-})  = - g(v_{-}, \nabla^{-}_{X} X_{-}) =0, 
$$
and we obtain that $\nabla_{X}^{-} v_{-} $ is a section of $T^{(1,0)}_{J_{-}} M$, i.e. $\nabla^{-}$ preserves $T^{(1,0)}_{J_{-}} M.$   
A similar argument 
which uses  the second relation (\ref{gr-3}) and the definition  (\ref{nabla-1-def})  of $\nabla^{+}$ 
shows that $\nabla^{+}$ preserves $T^{(1,0)}_{J_{+}} M.$  
\end{proof}

\begin{cor}\label{ex-4-dim} In the setting of Theorem \ref{1-even}, assume that $M$ is four dimensional. 
Then $( \mathcal{G},  \mathcal F )$ is  $B_{4}$-generalized  pseudo-K\"{a}hler  structure if and only if
the following conditions hold:

i) the covariant derivatives of $X_{\pm}$ with respect to the Levi-Civita connection $\nabla$ of $g$ are given by
\begin{align}
\nonumber& \nabla_{X} X_{+} = -\frac{1}{2}  (i_{X_{+}} H)(X) + c_{+} F(X),\\
\label{nabla-cov-even}&\nabla_{X} X_{-} = -\frac{1}{2} ( i_{X_{-}} H)(X) - J_{+} F(X).
\end{align}
for any $X\in {\mathfrak X}(M)$.\

ii) the connection
$D^{-} _{X}:=\nabla_{X} +\frac{1}{2} H(X)$
preserves $J_{-}$;\

iii) $F$ is of type $(1,1)$ with respect to $J_{-}$, $i_{X_{+}} F = d c_{+}$ and 
\begin{equation}\label{c-plus-ct}
H(X_{+} , X_{-}) = c_{+} F(X_{-}) + J_{+}  F(X_{+}).
\end{equation}
\end{cor}

\begin{proof} 
Since 
$M$ has dimension four, 
$\mathrm{rank}\, T_{J_{+}}^{(1,0)}M =1$ and  $\mathrm{rank}\, T_{J_{-}}^{(1,0)}M =2$.
Like in the proof of Corollary \ref{ex-3-dim}, 
one can show that $\mathrm{rank}\, T_{J_{+}}^{(1,0)}M =1$ implies that 
$D^{+}$ preserves $T_{J_{+}}^{(1,0)} M$.  We deduce that the conditions from Theorem \ref{1-even} 
(see also Lemma \ref{exchange}) 
reduce to the conditions
from Corollary \ref{ex-4-dim}.
\end{proof}

\section{Examples over Lie groups }\label{examples}

In order to illustrate our theory,  in this section we construct examples 
of  $B_{n}$-generalized pseudo-K\"{a}hler structures on Courant algebroids   of type $B_{n}$ over Lie groups of dimension two, three and four.

\begin{defn} 
A Courant algebroid   $E_{H, F}$  of type $B_{n}$ over a Lie group  $G$ is called {\cmssl  left-invariant} if the forms 
$H\in \Omega^{3}(G)$ and $F\in \Omega^{2}(G)$ are left invariant (recall that $dH + F\wedge F =0)$. 
A $B_{n}$-generalized pseudo-K\"{a}hler structure is  {\cmssl left-invariant}  if it is defined on a 
left-invariant Courant algebroid  and 
its components are left invariant tensor fields. 
\end{defn}

We identify as usual  left-invariant tensor fields on  a Lie group $G$ with tensors on the Lie algebra $\mathfrak{g}$ of $G$. In particular,
the forms   $H$ and $F$
 which define a left-invariant Courant algebroid  $E_{H, F}$ over $G$, 
 as well as the components of a left invariant $B_{n}$-generalized K\"{a}hler structure on $E_{H, F}$,  will be considered as 
tensors on the Lie algebra  $\mathfrak{g}$.

 \subsection{The case $\mathrm{dim}\, G=2$}\label{sect-dim-2}

 Let $G$ be a $2$-dimensional  Lie group 
 with Lie algebra $\mathfrak{g}$ and  $(\mathcal{G}, \mathcal F)$ a  left-invariant $B_{2}$-generalized 
 pseudo-K\"{a}hler structure on a Courant algebroid $E_{F}:= E_{0 , F}$ of type $B_{2}$ over $G$, with components
 $(g,  J_{+}, J_{-}, X_{+}, X_{-}, c_{+}) $ such that $X_{\pm}$ are non-null. 
 Recall that a $2$-dimensional Lie group with a left-invariant metric admits a (non-zero) left-invariant Killing field if and only if it is abelian.
 Since $X_{+}$ is a Killing field, we deduce that $G$ is abelian.
 From  Corollary \ref{ex-2-dim},  
 there is a 
 $g$-orthonormal basis $\{ v_{1}, v_{2} \}$ of $\mathfrak{g}$ 
  such that 
 \begin{equation}
 X_{+} = y v_{2}, \ X_{-} = y v_{1},\ 
J_{-} v_{1} = \epsilon_{0} v_{2},\ J_{-} v_{2} = - \epsilon_{0}  v_{1},\ J_{+} = \epsilon_{0} \epsilon_{+} (1- \epsilon y^{2})^{1/2} J_{-}, 
 \end{equation} 
where  
 $\epsilon_{0}, \epsilon_{+} \in \{ \pm 1 \} $ and $y\in \mathbb{R}\setminus \{ 0\} $
is such that  $\epsilon y^{2} \leq 1$, where 
$\epsilon := g(v_{i}, v_{i})\in \{ \pm 1\}$. The condition $i_{X_{+}} F = d c_{+}$ together with $c_{+}$ constant imply $F=0.$

\subsection{The case $\mathrm{dim}\, G = 3$}

Let $G$ be a  $3$-dimensional  Lie group with Lie algebra $\mathfrak{g}.$ 
 A  (non-degenerate) metric $g$ on   
$\mathfrak{g}$ defines  a canonical  operator  $L\in \mathrm{End}\,  (\mathfrak{g})$, unique up to multiplication by $\pm 1.$
By choosing an orientation on $\mathfrak{g}$, the operator $L $ relates the Lie bracket with the cross product
of $\mathfrak{g}$ determined by $g$ and the orientation, by
\begin{equation}\label{def-L}
[ u, v] = L (u\times v),\  \forall u, v\in \mathfrak{g}.
\end{equation} 
Reversing the orientation, the operator $L$ multiplies by $-1.$ 
It is well-known (see  \cite{milnor} 
for $g$ positive definite and  \cite{cortes-k} 
for $g$ of arbitrary signature) that  
 $G$  is unimodular 
(i.e.\ $\mathrm{trace}\, ( \mathrm{ad}_{x} ) =0$ for any $x\in \mathfrak{g}$)  
if and only if $L$ is self-adjoint.  In the next proposition we assume for simplicity that $L$ is diagonalizable. This is always the case
when $g$ is positive definite and $\mathfrak{g}$ is unimodular.

\begin{prop}\label{unimod-3} 
Let $G$ be a $3$-dimensional  unimodular Lie group, with Lie algebra $\mathfrak{g} .$ 
There is a left invariant $B_{3}$-generalized pseudo-K\"{a}hler structure $(\mathcal{G}, \mathcal F )$ 
on a  Courant algebroid  $E= E_{H, F}$ over $G$,  with 
components $(g,  J_{+}, J_{-}, X_{+}, X_{-})$, such that the operator $L\in \mathrm{End}\,  (\mathfrak{g})$ associated to $g$ is diagonalizable,
if and only if one of the following two cases holds:

i)  there is a  $g$-orthonormal  basis $\{ v_{1}, v_{2}, v_{3} \}$ of $\mathfrak{g}$, 
such that $g(v_{2}, v_{2}) =1$, 
 in which the Lie bracket of $\mathfrak{g}$ is given by
\begin{equation}
[v_{1}, v_{2} ] = \epsilon_{3} \lambda v_{3},\ [ v_{3}, v_{1} ] =0,\ [v_{2}, v_{3} ] = \epsilon_{1}\lambda v_{1},
\end{equation}
where  $\epsilon_{i}:= g( v_{i}, v_{i})\in \{ \pm 1\}$ 
($i\in \{ 1,3\}$), $\lambda \in \mathbb{R}\setminus \{ 0\}$, $X_{-} = v_{2}$ and $X_{+} = \pm X_{-}.$ 
In particular, $\mathfrak g$ is isomorphic to the Lie algebra of Killing fields of 
Euclidean or Minkowskian $2$-space, $\mathfrak{g}\cong \mathfrak{iso}(2)= \mathfrak{so}(2) \ltimes \mathbb{R}^2$ or $\mathfrak{g} \cong \mathfrak{iso}(1,1)=\mathfrak{so}(1,1)\ltimes \mathbb{R}^2$, depending on whether $\epsilon_1\epsilon_3=1$ or $\epsilon_1\epsilon_3=-1$, cf.\ Remark~\ref{change-basis-2}.

ii) $\mathfrak{g}$ is abelian, $g$ is an arbitrary  (non-degenerate)  metric on $\mathfrak{g}$ and $X_{\pm}\in \mathfrak{g}$ are arbitrary space-like unit vectors, i.e.\ $g(X_\pm ,X_\pm )=1$.

In both cases $ H =0$,  $F =0$ and $J_{\pm}\in \mathrm{End} (\mathfrak{g})$ are  arbitrary $g$-skew-symmetric endomorphisms, 
which satisfy $J_{\pm}X_{\pm} =0$
and   are complex structures on $X_{\pm}^{\perp}.$    
\end{prop}

\begin{proof}  
Assume that $\mathfrak{g} $ is non-abelian. Since $L$ is diagonalizable, 
there is a $g$-orthonormal basis 
$\{ v_{1}, v_{2}, v_{3}\}$ 
in which the Lie bracket of $\mathfrak{g}$ is given by 
\begin{equation}
[ v_{1}, v_{2} ] = \epsilon_{3}\lambda_{3} v_{3},\ 
[v_{3}, v_{1} ] = \epsilon_{2} \lambda_{2} v_{2},\ 
[v_{2}, v_{3} ]=\epsilon_{1}  \lambda_{1} v_{1},
\end{equation}
where $\lambda_{i} \in \mathbb{R}$ (not all of them zero) and 
$\epsilon_{i}:= g(v_{i}, v_{i})\in \{ \pm 1\}$. 
The Levi-Civita connection of $g$ is given by 
\begin{align}
\nonumber& \nabla_{v_{1}} v_{2} = \frac{\epsilon_{3}}{2} ( - \lambda_{1} +\lambda_{2} +\lambda_{3})v_{3},\
\nabla_{v_{2}} v_{1} = \frac{\epsilon_{3}}{2} ( -\lambda_{1} +\lambda_{2} -\lambda_{3})v_{3},\\
\nonumber&  \nabla_{v_{1}} v_{3} = \frac{\epsilon_{2} }{2} (  \lambda_{1} -\lambda_{2} -\lambda_{3})v_{2},\
\nabla_{v_{3}} v_{1} = \frac{\epsilon_{2} }{2} ( \lambda_{1} +\lambda_{2} -\lambda_{3})v_{2},\\
\label{levi-civita-3-non}&  \nabla_{v_{2}} v_{3} = \frac{\epsilon_{1}}{2} (  \lambda_{1} -\lambda_{2} +\lambda_{3})v_{1},\
\nabla_{v_{3}} v_{2} = \frac{\epsilon_{1}}{2} ( -\lambda_{1} -\lambda_{2} +\lambda_{3})v_{1},
\end{align}
and
\begin{equation}\label{levi-civita-3-non-0}
\nabla_{v_{i}} v_{i}=0,\ \forall 1\leq i\leq 3.
\end{equation}

Recall, from Corollary 
\ref{rem-comments}, that $X_{-}$ is a Killing field of norm one.  It turns out that $\mathfrak{g}$ admits  such a left-invariant
Killing field  in the following cases:

1) $\lambda_{1} \neq \lambda_{2}$,  $\lambda_{3} = \lambda_{1}$ and 
$\epsilon_{2} =1$  (up to permutation). Then any  space-like left-invariant unit Killing  field (in particular, $X_{-}$) 
 is of the form $X_{-} =\pm v_{2}$.    
Replacing $v_{2}$ by $- v_{2}$ (and leaving $v_{1}$ and $v_{3}$ unchanged), we may (and will) assume
that $X_{-} = v_{2}.$\ 

2) $\lambda_{1} = \lambda_{2} = \lambda_{3} \neq 0$. Any  left invariant vector field is 
Killing.\

Consider case 1) and let $\{ v_{1}^{*}, v_{2}^{*}, v_{3}^{*}\}$ be the dual basis of 
$\{ v_{1}, v_{2}, v_{3}\}$, i.e.\ $ v_{i}^{*} (v_{j}) = \delta_{ij}.$ The covectors $v_{i}^{*}$ correspond to $\epsilon_{i} v_{i}$ 
in the duality defined by $g$.
From  $i_{X_{-}} F =0$,  $X_{-} = v_{2}$ we deduce  that 
\begin{equation}\label{F-1-3}
F= F_{13} v_{1}^{*} \wedge v_{3}^{*}, 
\end{equation}
 where $F_{13}\in \mathbb{R}.$ 
Using  relations (\ref{levi-civita-3-non}) (with $\lambda_{3} =\lambda_{1}$) and relations (\ref{levi-civita-3-non-0}) 
we obtain from  
the first relation (\ref{gr-3}) with $X_{-} = v_{2}$ that  
\begin{equation}\label{H-1-3}
H= -  \lambda_{2} v_{1}^{*} \wedge v^{*}_{2} \wedge v^{*}_{3}.
\end{equation}
Recall now, from Corollary \ref{rem-comments}, that $X_{+}$ and $X_{-}$ commute.

When  $\lambda_{3} = \lambda_{1} \neq 0$ the conditions that $X_{\pm}$ have the same norm and  $\mathcal L_{X_{+}} X_{-} =0$
imply 
that  $X_{+} = \epsilon X_{-}$ with $\epsilon \in \{ \pm 1\}$. The  second relation (\ref{gr-3}) reduces to 
\begin{equation}\label{epsilon-h}
\epsilon H(X_{-}, X) = - J_{+} (F(X)),\ \forall X\in \mathfrak{g}. 
\end{equation}
Using that $J_{+}$ is $g$-skew-symmetric, $J_{+} X_{+} =0$ and $\mathrm{rank} ( \mathrm{Ker} \, J_{+}) =1$, 
together with (\ref{F-1-3}) and (\ref{H-1-3}) 
we 
obtain  from (\ref{epsilon-h}) that $H =0$, $F =0$ and $\lambda_{2} =0.$ This leads to the generalized pseudo-K\"{a}hler  structure from claim i).

When $\lambda_{3} = \lambda_{1} =0$, $X_{-}$ belongs to the center of $\mathfrak{g}$ and $\mathcal L_{X_{-}} X_{+} =0$ 
does not impose any restrictions on $X_{+}$ (as it happened when $\lambda_{3} =\lambda_{1} \neq 0$).  Since $g(X_{+}, X_{+}) =1$, the vector field  $X_{+}$ is of the form 
$X_{+} = \sum a_i v_{i}$, where
$a_1, a_2 ,a_3\in \mathbb{R}$, such that  $\epsilon_{1} a_{1}^{2} + a_{2}^{2} + \epsilon_{3} a_{3}^{2}=1$.
The second
relation (\ref{gr-3}) becomes
\begin{align}
\nonumber&  \sum a_i \nabla_Xv_i   -\frac{  \lambda_{2}}{2} 
(  a_1 v^{*}_{2} \wedge v^{*}_{3} -  a_2 v^{*}_{1} \wedge v^{*}_{3} +   a_3
v^{*}_{1} \wedge v^{*}_{2} ) (X)\\
\label{above-3}  &  = -  F_{13} J_{+} (v_{1}^{*}\wedge v^{*}_{3} )(X)
\end{align}
for any  $X\in\mathfrak{g}.$
Since $\mathfrak{g}$ is not abelian, $\lambda_{2} \neq 0.$ Relation (\ref{above-3}) with $X:= v_{1} $ and $X:= v_{2}$  implies $a_2=a_3=0$, i.e.\ 
$X_{+} = a_1 v_{1}.$ In particular, 
as $J_{+} X_{+} =0$, we obtain that 
$J_{_+} v_{1} =0.$   
Relation (\ref{above-3}) with $X := v_{3}$
implies that $F_{13} J_{+} v_{1} =   \epsilon_{1}a_1 \lambda_{2} v_{2}.$ 
Combined with $J_{+} v_{1} =0$, this implies that $a_1 =0$, which is a contradiction. Similar computations show that case 2) leads to
a contradiction as well.\

In the remaining case, i.e.\ when $\mathfrak{g}$ is abelian, Corollary~\ref{ex-3-dim} 
implies immediately that $H$ and $F$ are zero and the remaining components of the generalized almost pseudo-Hermitian structure are unconstrained.
\end{proof}

\begin{rem}\label{change-basis-2}{\rm Consider the $B_{3}$-generalized pseudo-K\"{a}hler structure
$(\mathcal{G}, \mathcal F  )$ 
 from Proposition  \ref{unimod-3} i) and the new basis 
$\{ w_{1}:= \frac{1}{\lambda} v_{1} , w_{2} := \frac{1}{\lambda} v_{2} ,  w_{3} := \frac{1}{\lambda} v_{3} 
\}$   of $\mathfrak{g} .$ Rescaling  $(\mathcal{G}, \mathcal F )$ by $\lambda^{2}$ 
(according to Corollary \ref{rescaling-1}) 
we obtain a 
$B_{3}$-generalized pseudo-K\"{a}hler structure  on the untwisted Courant algebroid  of type $B_{3}$ over $G$, with 
components 
$(\tilde{g},\tilde{J}_{+}, \tilde{J}_{-}, \tilde{X}_{+}, \tilde{X}_{-})$, such that $\{  w_{1} , w_{2}, w_{3} \}$ is  
$\tilde{g}$-orthogonal, and  
\begin{equation}
\tilde{g}( w_{i}, w_{i}) =  \epsilon_{i},\  \tilde{g}( w_{2}, w_{2} ) =1,\ 
\tilde{X}_{-} =  w_{2},\  \tilde{X}_{+} = \pm  \tilde{X}_{-},
\end{equation}
where $\epsilon_{i} \in \{ \pm 1\}$ for $i\in \{ 1, 3\} .$
 As above, $\tilde{J}_{\pm}\in \mathrm{End}\, ( \mathfrak{g})$
are  arbitrary $g$-skew-symmetric endomorphisms, such that 
$\tilde{J}_{\pm } \tilde{X}_{\pm }=0$ and $\tilde{J}_{\pm}$ are complex structures on
$\tilde{X}_{\pm}^{\perp}.$  In the new basis the structure constants  of $\mathfrak{g}$  take the standard form (as in \cite{milnor}):
\begin{equation}
[ w_{1} , w_{2} ] =\epsilon_{3}  w_{3},\ 
[ w_{2} , w_{3} ] = \epsilon_{1} w_{1},\ [ w_{3}, w_{1} ]=0.
\end{equation}}
\end{rem}

Let  $G$ be a $3$-dimensional  non-unimodular Lie group, with Lie algebra $\mathfrak{g}.$ 
Since $\mathfrak{g}$ is $3$-dimensional, its unimodular kernel
\begin{equation}
\mathfrak{g}_{0}:= \{ x\in \mathfrak{g}:\ \mathrm{trace} (\mathrm{ad}_{x}) =0 \}
\end{equation}
is $2$-dimensional (and unimodular), hence abelian.  Choose  a basis $\{ v_{2}, v_{3 }\}$ of
$\mathfrak{g}_0$ and a vector $v_{1}\notin\mathfrak{g}_{0}$. 
In the basis $\{ v_{1}, v_{2}, v_{3} \}$ the Lie bracket of 
$\mathfrak{g}$ is given by 
\begin{equation}\label{lie-form}
[v_{1}, v_{2} ]= \alpha v_{2} + \beta v_{3},\ [v_{1}, v_{3} ] = \gamma v_{2} +\delta v_{3} ,\ [ v_{2}, v_{3} ] =0,
\end{equation}
where 
$\alpha , \beta , \gamma , \delta \in \mathbb{R}$ and 
$\alpha +\delta \neq 0$. Up to a multiplicative factor the constants
 $\alpha , \beta , \gamma , \delta $ are independent of the choice of $v_{1}.$

\begin{prop}\label{non-mod-3}  
Let $G$ be a $3$-dimensional  non-unimodular Lie group, with Lie algebra $\mathfrak{g} $
and unimodular kernel 
$\mathfrak{g}_{0}.$ 
There is a  left invariant  $B_{3}$-generalized pseudo-K\"{a}hler structure $(\mathcal{G},  \mathcal F )$ 
on a  Courant algebroid  $E= E_{H, F}$ of type $B_{3}$ over $G$,  with 
components $(g,  J_{+}, J_{-}, X_{+}, X_{-})$, such that 
$\mathfrak{g}_{0}^{\perp}\cap \mathfrak{g}_{0} =0$, 
if and only if   $\mathfrak g$ is isomorphic to $\mathbb{R}\oplus \mathfrak{sol}_2$, where 
$\mathfrak{sol}_2$ is the unique non-abelian Lie algebra of dimension $2$. 
If $\mathfrak{g}$ is isomorphic to   $\mathbb{R}\oplus \mathfrak{sol}_2$, then there is a basis 
$\{ w_{1}, w_{2}, w_{3} \}$ of $\mathfrak{g}$ 
in which the Lie brackets take the form
\begin{equation}\label{brackets-w}
[w_{1}, w_{2} ] = w_{2},\ [ w_{1}, w_{3} ] =  [w_{2}, w_{3} ] = 0,
\end{equation}
the metric $g$ is given by
\begin{equation}
g (w_{1}, w_{1}) = \frac{\epsilon}{\delta^{2}},\ g(w_{2}, w_{2})= \epsilon' ,\ g( w_{3}, w_{3}) = 1,\ g(w_{i},w_{j}) =0\, \forall i\neq j, 
\end{equation}
where $\delta \in \mathbb{R}\setminus \{ 0\}$,  $\epsilon, \epsilon' \in \{ \pm 1\}$ are arbitrary, 
$X_{-} = w_{3}$,   $X_{+} = \pm X_{-}$ 
and $J_{\pm}\in \mathrm{End} (\mathfrak{g})$ are arbitrary $g$-skew-symmetric endomorphisms, which satisfy $J_{\pm}X_{\pm} =0$
and are complex structures on $X_{\pm}^{\perp}.$    Moreover, $H=0$ and $F=0$.
\end{prop}

\begin{proof}
Let  $g$ be a left-invariant metric on $G$ such that 
$\mathfrak{g}_{0}^{\perp}\cap \mathfrak{g}_{0} =0$. Choose a $g$-orthonormal basis $\{ v_{1}, v_{2}, v_{3}\}$ of $\mathfrak{g}$ such that
$v_{2}, v_{3} \in \mathfrak{g}_{0}$. 
The Lie bracket of $\mathfrak{g}$ takes the form (\ref{lie-form}) 
in the basis $\{ v_{1}, v_{2}, v_{3} \}$ 
and
 the Levi-Civita connection $\nabla$ of $g$ is given by 
 \begin{align}
 \nonumber& \nabla_{v_{1}} v_{1} =0,\ \nabla_{v_{2}} v_{2} = \epsilon_{1} \epsilon_{2} \alpha v_{1},\ \nabla_{v_{3}} v_{3} = \epsilon_{1} \epsilon_{3} \delta v_{1}\\
 \nonumber & \nabla_{v_{1}} v_{2} = \frac{1}{2} (\beta -\epsilon_{2} \epsilon_{3}\gamma ) v_{3},\  \nabla_{v_{2}} v_{1} = -\frac{1}{2} ( \epsilon_{2} \epsilon_{3} \gamma + \beta ) v_{3}- \alpha v_{2}\\
 \nonumber& \nabla_{v_{3}} v_{1} = -\frac{1}{2} (\epsilon_{2} \epsilon_{3}  \beta +\gamma  ) v_{2} -\delta v_{3},\ \nabla_{v_{1}} v_{3} = \frac{1}{2} (\gamma - \epsilon_{2} \epsilon_{3}  \beta )
 v_{2}\\
 \label{levi-civita-3} &  \nabla_{v_{3}} v_{2} =  \nabla_{v_{2}} v_{3} = \frac{\epsilon_{1} }{2} ( \epsilon_{2} \gamma + \epsilon_{3} \beta ) 
 v_{1}, 
 \end{align}
 where $\epsilon_{i} := g (v_{i}, v_{i}) \in \{ \pm 1\} .$ 

From Corollary 
\ref{rem-comments},  $X_{-}$ is a Killing  field. It turns out that $g$ admits a non-zero left-invariant Killing  field only in the following cases:

1)   $\alpha =0$, $\delta \neq 0$  and $\beta \gamma =0$. Any non-zero left-invariant Killing field (in particular $X_{-}$)  is of the form   $X_{-}= b ( v_{2} -\frac{\beta}{\delta } v_{3})$, where $b\in \mathbb{R}\setminus \{ 0\}$; or

2)   $\alpha \neq 0$, 
$\alpha + \delta \neq 0$ and 
$\gamma \beta = \delta \alpha$. Any  non-zero left-invariant Killing  field (in particular $X_{-}$) 
is of the form 
 $X_{-} =
c( -\frac{\gamma}{\alpha} v_{2} + v_{3})$, where $c\in \mathbb{R}\setminus \{ 0\}$.\

Consider case 1)  
and let  $\{ v_{1}^{*}, v_{2}^{*}, v_{3}^{*}\}$  be  the dual basis.
Relation $i_{X_{-}} F =0$ implies that
\begin{equation}\label{F-3}
F = F_{13} ( \frac{\beta}{\delta} v^{*}_{1} \wedge v^{*}_{2} + v_{1}^{*}\wedge v^{*}_{3} ),
\end{equation}
where $F_{13}\in \mathbb{R}$.  We write $H = H_{123} v_{1}^{*} \wedge v_{2}^{*} \wedge v_{3}^{*}$, where $H_{123}\in \mathbb{R}.$ 
Using  relations (\ref{levi-civita-3}) (with $\alpha =0$) we obtain that the first relation (\ref{gr-3}) is equivalent to 
$H_{123} =  \epsilon_{2} \gamma - \epsilon_{3} \beta $, i.e.
\begin{equation}
H =  ( \epsilon_{2} \gamma - \epsilon_{3} \beta )v^{*}_{1} \wedge v^{*}_{2} \wedge v^{*}_{3}.
\end{equation}
For the second relation
(\ref{gr-3}), let 
$X_{+} :=  \sum_{i} a_{i} v_{i}$, where $a_{1}, a_{2}, a_{3} \in \mathbb{R}.$ With 
this notation, the second relation (\ref{gr-3}) is equivalent to
\begin{align}
\nonumber& J_{_+} F (v_{1} )= (\beta - \epsilon_{2} \epsilon_{3} \gamma ) (\epsilon_{2} \epsilon_{3}  a_{3} v_{2} - a_{2} v_{3})\\
\nonumber& J_{+} F(v_{2}) = \beta ( -  \epsilon_{1} \epsilon_{3} a_{3}  v_{1} + a_{1} v_{3})\\
\nonumber& J_{+} F (v_{3}) = -\epsilon_{1}  (  \epsilon_{2}  a_{2} \gamma + \epsilon_{3}   a_{3} \delta ) v_{1} +  a_{1} \gamma v_{2} + a_{1} \delta v_{3},
\end{align}
or, using (\ref{F-3}), to 
\begin{align}
\nonumber& F_{13} J_{+} (  \epsilon_{2} \frac{\beta}{\delta} v_{2}+  \epsilon_{3}v_{3} ) = (\beta - \epsilon_{2} \epsilon_{3} \gamma ) (  \epsilon_{2} \epsilon_{3}  a_{3} v_{2} - a_{2}  v_{3} )\\
\nonumber& \frac{\beta}{\delta} F_{13} J_{+} ( v_{1}) = \beta (  \epsilon_{3} a_{3}  v_{1} - \epsilon_{1}  a_{1} v_{3} )\\
\label{f-1-3}& F_{13} J_{+} ( v_{1})  =  ( \epsilon_{2}  a_{2} \gamma +  \epsilon_{3}  a_{3} \delta ) v_{1} -  \epsilon_{1}  a_{1}\gamma v_{2} -  \epsilon_{1} a_{1} \delta v_{3}.
\end{align}

Recall now that $J_{+} \in \mathrm{End} (\mathfrak{g})$ is $g$-skew-symmetric, $J_{+} X_{+} =0$ and
$J_{+} $ is a complex structure on the orthogonal $X_{+}^{\perp}$.  
These conditions combined with (\ref{f-1-3}) 
 imply that $F_{13} = H_{123}=\beta = \gamma =  a_{1} = a_{3} =0$. 
 To summarize: case 1) provides  a  basis $\{ v_{1} , v_{2}, v_{3}\}$ of
 $\mathfrak{g} $ and 
 a  one parameter family  (indexed by $\delta\in \mathbb{R}\setminus \{ 0\}$, see below) of $B_{3}$-generalized pseudo-K\"{a}hler structures 
 on the untwisted Courant algebroid ($H=0$, $F=0$),  
 with components
 $(g, J_{+}, J_{-}, X_{+}, X_{-})$, 
 such that the basis $\{ v_{1}, v_{2}, v_{3}\}$ is $g$-orthonormal, 
 $ g (v_{2}, v_{2}) =1$,  the Lie bracket of $\mathfrak{g}$ is given by  
\begin{equation}\label{bracket-1}
[v_{1}, v_{2} ] = 0,\ [ v_{1}, v_{3} ] = \delta  v_{3},\ [v_{2}, v_{3} ] = 0,
\end{equation}
where $\delta \in \mathbb{R}\setminus \{ 0\}$, 
$X_{-} = v_{2}$ and  $X_{+} = \pm X_{-}$. 
In the new basis 
$\{ w_{1} := \frac{1}{\delta } v_{1},  w_{2} := v_{3}, w_{3} := v_{2} \}$ 
this family takes the form described in the statement of the proposition, where $\epsilon=\epsilon_1$ and $\epsilon'=\epsilon_3$. \

Similar arguments show that case 2)  with $\gamma\neq 0$ leads to  
a basis $\{ v_{1}, v_{2}, v_{3}\}$ of $\mathfrak{g}$ and 
a  family of $B_{3}$-generalized pseudo-K\"{a}hler structures  on the untwisted Courant algebroid $(H=0$, $F=0$), with components $(g, J_{+}, J_{-}, X_{+}, X_{-})$,   such that  $\{ v_{1}, v_{2}, v_{3} \}$
is $g$-orthonormal,  the Lie bracket takes the form
\begin{equation}\label{bracket-2}
[v_{1}, v_{2} ] = \alpha v_{2} + \gamma  \epsilon_{2} \epsilon_{3} v_{3},  \ [ v_{1}, v_{3} ] = \gamma v_{2} +\epsilon_{2} \epsilon_{3}  \frac{\gamma^{2}}{\alpha}
v_{3},\ [v_{2}, v_{3} ] = 0,
\end{equation}
where 
$\alpha  , \gamma   \in \mathbb{R}\setminus \{ 0\}$,  $\epsilon_{i}:= g(v_{i}, v_{i}) \in \{ \pm 1\}$, 
$X_{-} = c( -\frac{\gamma}{\alpha} v_{2} +  v_{3})$ and $X_{+} = \pm X_{-}$, where  $c\in \mathbb{R} $ satisfies 
$c^{2} ( \alpha^{2} \epsilon_{3} + \gamma^{2} \epsilon_{2})=\alpha^{2}$.
In the new basis  
$$\left\{ w_{1} := \frac{\epsilon_3c^2}{\alpha} v_{1}, w_{2} : =\frac{\epsilon_{3}c}{\alpha} (\alpha v_{2} +  \epsilon_{2} \epsilon_{3} \gamma v_{3} ), w_{3} := 
-\frac{c}{\alpha}(\gamma v_2-\alpha v_3)\right\} $$  the Lie brackets take the form
(\ref{brackets-w}) and the metric $g$ and vector fields $X_{\pm}$  are given by
\begin{align}
\nonumber&  g(w_{1}, w_{1}) =\frac{c^4}{\alpha^2}\epsilon_{1},\ g( w_{2}, w_{2}) =  \epsilon_{2} \epsilon_{3} ,\ g( w_{3}, w_{3}) = 1 ,\\
\nonumber& 
g(w_{1}, w_{2}) = g(w_{1}, w_{3}) =g( w_{2}, 
w_{3}) =0,\\
\label{resc-ii}& X_{-} =  w_{3},\ X_{+} = \pm X_{-}.
\end{align}
Letting   $\delta$ such that  $\delta^2=\alpha^2/c^4$ we arrive again at the $B_{3}$-generalized pseudo-K\"{a}hler structures described in the statement of the proposition,
 where now $\epsilon=\epsilon_1$ and 
$\epsilon'=\epsilon_{2}\epsilon_{3}$.\

Case 2) with $\gamma =0$ leads 
to  the  family of $B_{3}$-generalized pseudo-K\"{a}hler structures obtained  in case 1), but with $v_{2}$ and $v_{3}$ interchanged. 
Therefore, they provide  no further  $B_{3}$-generalized pseudo-K\"{a}hler structures besides  those described above.
\end{proof}

The next corollary summarizes our results 
from this section in the positive definite case.

\begin{cor} \label{classification}
Let $(\mathcal{G}, \mathcal F )$ be a left-invariant $B_{3}$-generalized K\"{a}hler structure on a Courant algebroid
$E = E_{H, F}$ of type $B_{3}$  over a $3$-dimensional Lie group $G$ with Lie algebra $\mathfrak{g}.$ 
Let 
$(g,  J_{+}, J_{-}, X_{+}, X_{-})$ be the components of $(\mathcal{G}, \mathcal F).$ 
Up to rescaling of $(\mathcal{G}, \mathcal F )$
one  of the following situations holds:\

i)  there is a $g$-orthonormal basis $\{ w_{1}, w_{2}, w_{3} \}$ of $\mathfrak{g}$ 
in which the Lie brackets   take the form
\begin{equation}
[ w_{1} , w_{2} ] =  w_{3},\ 
[  w_{2} , w_{3} ] =  w_{1},\ [ w_{3}, w_{1} ]=0
\end{equation}
and  $X_{-} = w_{2}$, $ X_{+} = \pm  X_{-}.$\

 ii) there is a $g$-orthogonal basis 
$\{ w_{1}, w_{2}, w_{3} \}$ of $\mathfrak{g}$ 
in which the Lie brackets  take the form 
\begin{equation}
[w_{1}, w_{2} ] = w_{2},\ [w_{1},w_{3} ] =  [w_{2}, w_{3} ] = 0,
\end{equation}
and 
\begin{equation}
g (w_{1}, w_{1}) = \frac{1}{\delta^{2}},\ g(w_{2}, w_{2})= 1 ,\ g(w_{3}, w_{3}) = 1,
\end{equation}
where $\delta \in \mathbb{R}\setminus \{ 0\}$,  
$X_{-} = w_{3}$,   $X_{+} = \pm X_{-}$.\

iii) the Lie algebra  $\mathfrak{g}$ is abelian,  $g$ is any Riemannian  metric on $\mathfrak{g}$ and $X_{\pm}$ are arbitrary vectors from
$\mathfrak{g}$, of norm one. \

 In all cases above   $H = 0$, $F =0$ and $J_{\pm}\in \mathrm{End}\, ( \mathfrak{g})$
are   $g$-skew-symmetric endomorphisms, with the properties that 
$J_{\pm } X_{\pm }=0$ and $J_{\pm}\vert_{X_{\pm}^{\perp}}$ are complex structures.  
\end{cor}

\begin{proof} The claim follows from Propositions \ref{unimod-3} and \ref{non-mod-3}
(see  also Remark \ref{change-basis-2}),  together with 
the observation that 
if $g$ is positive definite then the operator $L$ from Proposition \ref{unimod-3} is diagonalizable and 
the assumption $\mathfrak{g}_{0}^{\perp}\cap \mathfrak{g}_{0} =0$ from Proposition \ref{non-mod-3}  is satisfied.
\end{proof}

\subsection{The case $\mathrm{dim}\, G =4$}

Let $G$ be a $4$-dimensional Lie group with Lie algebra $\mathfrak{g}.$ 
We assume that    
$\mathfrak{g}$  is of the form
\begin{equation}\label{form-g}
\mathfrak{g} = \mathfrak{u} + \mathfrak{g}_{0}
\end{equation}
where $\mathfrak{g}_{0}$ is a $3$-dimensional  unimodular non-abelian Lie algebra and
$\mathfrak{u}$ is  $1$-dimensional and acts on $\mathfrak{g}_{0}$ as a  ($1$-dimensional) family of derivations. 

Note that such a Lie algebra $\mathfrak{g}$ is unimodular if and only if $\mathrm{tr}\, ad_X=0$
for a non-zero element $X\in \mathfrak u$.
\begin{exa}
Every non-unimodular Lie algebra $\mathfrak{g}$ admits a (unique) codimension one 
unimodular ideal $\mathfrak{g}_0= \mathrm{Ker}\, (\mathrm{tr}\circ ad)$, called the \emph{unimodular kernel} of 
$\mathfrak g$. Choosing a complementary line 
$\mathfrak{u}$ we arrive at a decomposition $\mathfrak{g}=\mathfrak{u} + \mathfrak{g}_{0}$. 
So the assumptions of this section are satisfied if the unimodular kernel of $\mathfrak g$ is not abelian. 
\end{exa}

As an application of Corollary \ref{ex-4-dim}, we now describe 
a  class (called {\cmssl adapted})  
of left invariant  $B_{4}$-generalized pseudo-K\"{a}hler structures 
over $G$. 

\begin{defn}\label{adapted} A  left-invariant $B_{4}$-generalized pseudo-Hermitian structure $(\mathcal{G}, \mathcal F )$ on a Courant algebroid
over $G$, 
with components $(g,   J_{+}, J_{-}, X_{+}, X_{-}, c_{+})$, 
is called  {\cmssl adapted}  if the decomposition $\mathfrak{g} = \mathfrak{u} + \mathfrak{g}_{0}$ is orthogonal with respect to 
$g$, the operator $L$ associated to   $(\mathfrak{g}_{0}, g\vert_{\mathfrak{g}_{0}\times \mathfrak{g}_{0} })$ 
is diagonalizable,  $J_{-} (\mathfrak{u})$ is included in  an eigenspace of $L$ and  $\mathfrak{u}$ and $X_{\pm}$ are non-null
(i.e.\ $c_{+} \in \mathbb{R}\setminus \{ \pm 1\}$). 
 \end{defn} 

If   $(g,  J_{+}, J_{-},  X_{+}, X_{-}, c_{+})$ are the components of  an adapted $B_{4}$-generalized pseudo-Hermitian  structure, then
there is a 
$g$-orthonormal basis   of $\mathfrak{g}$  (called {\cmssl adapted})  of the form
$\{ u, e_{1}, e_{2}, e_{3} \}$, where $u\in \mathfrak{u}$
and $e_{i} \in \mathfrak{g}_{0}$,
\begin{equation}\label{j-m}
J_{-} u = e_{1} ,\ J_{-} e_{2}=  e_{3},
\end{equation}
in which the Lie bracket takes the form
\begin{equation}\label{form-1}
[e_{1}, e_{2} ] = \epsilon_{3} \lambda_{3} e_{3},\ [e_{2}, e_{3} ] = \epsilon_{1}  \lambda_{1} e_{1},\ [e_{3}, e_{1} ] =\epsilon_{2} \lambda_{2} e_{2},\ 
[ u, e_{i}  ]= \sum_{j=1}^{3} a_{ij} e_{j},
\end{equation} 
where  
$a_{ij} \in \mathbb{R}$ and $\lambda_{i} \in \mathbb{R}$ 
(at least one non-zero).
Since $\mathrm{ad}_{u}$ is a derivation of
$\mathfrak{g}_{0}$,
\begin{align}
\nonumber& \lambda_{1} (-  a_{11} + a_{22} + a_{33}) =0\\
\nonumber& \lambda_{2}  (a_{11} - a_{22} + a_{33}) =0\\
\nonumber& \lambda_{3} ( a_{11} + a_{22} - a_{33}) =0\\
\label{cond-deriv-a}& \epsilon_{i}  \lambda_{i} a_{ij} = -\epsilon_{j} \lambda_{j} a_{ji},\ \forall i\neq j,
\end{align}
where  $\epsilon_{i} := g( v_{i}, v_{i})\in \{ \pm 1\}$ for any $i\in \{ 1, 2, 3\} .$  Remark that $\epsilon_{0} := g (u, u)= \epsilon_{1}$
and $\epsilon_{2} = \epsilon_{3}$ 
since $J_{-}$ is $g$-skew-symmetric.\\

\begin{prop}\label{4c-n0}
There is an adapted  $B_{4}$-generalized  pseudo-K\"{a}hler structure $(\mathcal{G}, \mathcal F )$ 
on a Courant algebroid  $E = E_{H, F}$  over $G$,  with 
components $(g,  J_{+}, J_{-}, X_{+}, X_{-}, c_{+})$,  such that $c_{+}\neq 0$, 
if and only if 
there is a $g$-orthonormal basis of $\{ u, e_{1}, e_{2}, e_{3} \}$ of $\mathfrak{g}$ such that
\begin{align}
\nonumber& [e_{1}, e_{2} ] = \epsilon_{3} \lambda   e_{3},\ [e_{2}, e_{3} ] =  0, \ [e_{3}, e_{1} ] =\epsilon_{2}  \lambda e_{2},\\
\label{form-1-rez}& [u, e_{1} ] =0,\ [ u, e_{2}  ]=  \beta  e_{3},\ [ u, e_{3} ] = - \beta  e_{2},
\end{align} 
where  
$\lambda \in \mathbb{R}\setminus \{ 0\}$, $\beta  \in \mathbb{R}$,   $\epsilon_{i} := g(e_{i}, e_{i}) \in \{ \pm 1\}$.  
The remainig data are given by: $X_{+} = a u + b e_{1}$, $X_{-} = \tilde{a} u + \tilde{b} e_{1}$, $J_{+} X_{+} = - c_{+} X_{-}$, $J_{+} X_{-} = c_{+} X_{+}$, where 
$a\in \mathbb{R}\setminus \{ 0\}$, 
 $b, \tilde{a}, \tilde{b} \in \mathbb{R}$ are such that
\begin{equation}\label{constraint:eq}
\epsilon_{1}  ( a^{2} + b^{2}) = \epsilon_{1} ( \tilde{a}^{2} + \tilde{b}^{2}) = 1- c_{+}^{2},\ a \tilde{a} + b \tilde{b} =0,
\end{equation}
$c_{+} \in \mathbb{R}\setminus \{ -1, 0 , 1\}$, the complex structure $J_{-}$ 
is given by (\ref{j-m}) and $J_{+} \vert_{\mathrm{span}\{ e_{2}, e_{3} \}}$ is any   $g$-skew-symmetric complex structure.
Finally, we have $H =0$ and $F =0$. 
 \end{prop} 

We divide the proof of Proposition~\ref{4c-n0} 
into several lemmas.  Let $(\mathcal{G}, \mathcal F)$ be an adapted $B_{4}$-generalized pseudo-Hermitian structure on $G$, 
with components
$(g,  J_{+}, J_{-}, X_{+}, X_{-} , c_{+})$,   where  $c_{+} \in \mathbb{R}\setminus \{ \pm1\}$ is arbitrary.  As above,  let $\{ u, e_{1}, e_{2}, e_{3} \}$ 
be an adapted basis, $\epsilon_{i} := g(e_{i}, e_{i})$ and $\epsilon_{0}:= g ( u, u).$

\begin{lem}\label{levi-4-mod}
The Levi-Civita connection of $g$ is given by
\begin{align}
\nonumber&\nabla_{e_{1}} e_{2} = \frac{\epsilon_{3} }{2} ( -\lambda_{1} +\lambda_{2} +\lambda_{3} ) e_{3} +\frac{\epsilon_{0} }{2} 
( \epsilon_{2} a_{12} +  \epsilon_{1} a_{21} ) u\\
\nonumber&\nabla_{e_{2}} e_{1} =  \frac{\epsilon_{3} }{2} (-  \lambda_{1} + \lambda_{2} -\lambda_{3} )e_{3} 
+\frac{\epsilon_{0}}{2} (  \epsilon_{2} a_{12} + \epsilon_{1} a_{21}) u\\
\nonumber& \nabla_{e_{1}} e_{3} = \frac{\epsilon_{2} }{2} ( \lambda_{1} -\lambda_{2} -\lambda_{3}) e_{2} +\frac{\epsilon_{0} }{2} 
( \epsilon_{3} a_{13} + \epsilon_{1} a_{31}) u\\
\nonumber& \nabla_{e_{3}} e_{1} = \frac{\epsilon_{2} }{2} ( \lambda_{1} + \lambda_{2} -\lambda_{3}) e_{2} +\frac{\epsilon_{0} }{2} 
( \epsilon_{3} a_{13} + \epsilon_{1} a_{31}) u\\
\nonumber& \nabla_{e_{2}} e_{3} = \frac{\epsilon_{1} }{2} ( \lambda_{1} -\lambda_{2} + \lambda_{3}) e_{1} +\frac{\epsilon_{0} }{2} (
\epsilon_{3}  a_{23} +  \epsilon_{2} a_{32}) u\\
\nonumber& \nabla_{e_{3}} e_{2} = \frac{\epsilon_{1} }{2} ( - \lambda_{1} - \lambda_{2} + \lambda_{3}) e_{1} +\frac{\epsilon_{0} }{2} (\epsilon_{3}  a_{23} +  \epsilon_{2} a_{32}) u\\
\nonumber& \nabla_{u} e_{i}  =\frac{1}{2} \sum_{j} ( a_{ij} -  \epsilon_{i} \epsilon_{j} a_{ji}) e_{j},\ \nabla_{e_{i}} e_{i} = \epsilon_{0} \epsilon_{i} a_{ii} u\\
\nonumber& \nabla_{e_{i}} u  = -\frac{1}{2} \sum_{j} ( a_{ij}+  \epsilon_{i} \epsilon_{j} a_{ji}) e_{j},\ \nabla_{u} u =0.
\end{align}
\end{lem}

Let   $\{ u^{*}, e_{1}^{*} , e_{2}^{*}, e_{3}^{*} \}$ be  the dual basis
of $ \{ u, e_{1}, e_{2}, e_{3} \}$, i.e.
$$
e_{i}^{*} (e_{j}) = \delta_{ij},\ e_{i}^{*} (u) =0,\ 
u^{*} (u) =1,\ u^{*} (e_{i}) =0,\ \forall i, j.
$$
We write the left-invariant forms  $H$ and $F$ as
\begin{align}
\nonumber& H = H_{123} e_{1}^{*} \wedge e_{2}^{*} \wedge e_{3}^{*} + \frac{1}{2} H_{ij} u^{*}\wedge e_{i}^{*}\wedge e_{j}^{*} \\
\label{H-F-4-uni}& F =\frac{1}{2} F_{ij} e_{i}^{*}\wedge e_{j}^{*}  +  F_{i} u^{*}\wedge e_{i}^{*},
\end{align}
where 
$H_{123}, H_{ij}, F_{ij}, F_{i} \in \mathbb{R}$,  $H_{ij} = - H_{ji}$, $F_{ij} = - F_{ji}$  for any $i, j$,  and to simplify notation  
we omitted the summation signs.  As the pair $(H, F)$ defines a Courant algebroid of type $B_{4}$, the coefficients of 
$H$ and $F$ are subject to various constraints
which come from  $ dF =0$ and $ dH + F\wedge F =0.$

 \begin{lem}\label{F-cl} i) The equality $ dH + F\wedge F =0$ holds if and only if
 \begin{equation}\label{H-cl}
 H_{123}  (a_{11} + a_{22} + a_{33}) = 2( F_{1} F_{23} + F_{3} F_{12} + F_{2} F_{31}).
 \end{equation}
ii) The $2$-form $F$ is closed if and only if
 \begin{align}
\nonumber& \epsilon_{1}  \lambda_{1} F_{1} = F_{23} ( a_{22} + a_{33})  +  F_{21} a_{31} + F_{13} a_{21}\\
\nonumber& \epsilon_{2}  \lambda_{2} F_{2} = F_{31} ( a_{11} + a_{33})  +  F_{21} a_{32} +  F_{32} a_{12}\\
\label{F-cl}&\epsilon_{3}  \lambda_{3} F_{3} = F_{12} ( a_{11} + a_{22}) +  F_{32} a_{13} + F_{13} a_{23}.
\end{align}
\end{lem}

\begin{proof} 
The claims follow from a straightforward computation, which uses that the $1$-form 
 $u^{*}$ is closed and the exterior derivatives of the $1$-forms $e_{i}^{*}$ are given by
\begin{align}
\nonumber& d(e^{*} _{1} )= - \epsilon_{1} \lambda_{1} e_{2}^{*} \wedge e^{*}_{3} - a_{j1} u^{*}\wedge e^{*}_{j}\\
\nonumber& d (e_{2}^{*} )= - \epsilon_{2} \lambda_{2}  e^{*}_{3}\wedge e^{*}_{1} - a_{j2} u^{*}\wedge e^{*}_{j}\\
\nonumber& d(e^{*}_{3} )= - \epsilon_{3} \lambda_{3} e_{1}^{*}\wedge e^{*}_{2} - a_{j3} u^{*}\wedge e^{*}_{j}.
\end{align}
For instance, these equations imply
\[ dH = -\mathrm{tr} (A) H_{123}u^*\wedge  e_{1}^{*} \wedge e_{2}^{*} \wedge e_{3}^{*},\quad A:=(a_{ij}),\]
and comparing with 
\[ F \wedge F = 2 \left( \sum_{\mathfrak{S}} F_{ij}F_k \right) u^* \wedge  e_{1}^{*} \wedge e_{2}^{*} \wedge e_{3}^{*}\]
yields (\ref{H-cl}), where $\mathfrak{S}$ indicates the sum over cyclic permutations.
\end{proof}

We now apply Corollary \ref{ex-4-dim} with $c_{+}$  a non-zero constant. We consider each condition from this corollary separately. 

\begin{lem}\label{F-11}  The connection $D^{-} = \nabla +\frac{1}{2} H$ preserves $J_{-}$ if and only if
\begin{align}
\nonumber& a_{21} = a_{31}=0,\  a_{23} + a_{32} =0,\\
\nonumber& a_{22} - a_{33} =  \epsilon_2(\lambda_{3} -\lambda_{2}),\\
\nonumber& H_{23} =0,\ H_{12} = -  \epsilon_{2} a_{12},\ H_{13} = - \epsilon_{3} a_{13},\\
\label{cond-coeff} & H_{123} = 2 a_{22}\epsilon_{2}  - \lambda_{1} +\lambda_{2} -\lambda_{3}.
\end{align}
\end{lem}

\begin{proof} 
Recall the definition (\ref{j-m}) of $J_{-}.$ 
Using  Lemma \ref{levi-4-mod}, one can check that $(D^{-}_{X} J_{-})(u) =0$  for any $X\in \mathfrak{g}$ is equivalent to relations
(\ref{cond-coeff}). Moreover, these relations imply $ (D^{-}_{X} J_{-})(e_{i}) =0$ for any $i$.
\end{proof}

Since $c_{+}\neq 0$,  from the first relation (\ref{nabla-cov-even}) 
we obtain that 
\begin{equation}\label{F-concrete}
F = \frac{1}{2c_{+}} ( d X_{+} ^{\flat} + i_{X_{+}} H),
\end{equation}
where $X_{+}^{\flat}$ is the $1$-form $g$-dual to $X_{+}.$ In particular, $X_{+}$ is a Killing field.
On the other hand, remark that if $X$ is a Killing field of constant norm for a  pseudo-Riemannian metric $h$,  then $i_{X} d  X^{\flat} =0$
where $X^{\flat}$ is the $h$-dual to $X$.
We deduce that  the condition $i_{X_{+}} F =0$ 
from Corollary \ref{ex-4-dim} is satisfied, once we know that $X_{+}$ is  a Killing field.
The next lemma determines the conditions satisfied by the coefficients of Killing fields for  $g$.

\begin{lem}
Assume that $a_{21} = a_{31} =0$ and $a_{23} + a_{32} =0.$ A vector field  $X_{+}= au + be_{1} + ce_{2} + de_{3}$  is Killing for  $g$ if and only if
\begin{align}
\nonumber& b a_{11} = a a_{ii} = b a_{12} + c a_{22} + d a_{32}=0,\  \forall i\\
\nonumber& b a_{13} + c a_{23} + d a_{33} =0\\
\nonumber&  \epsilon_{2} a a_{12} + d(\lambda_{2} - \lambda_{1}) =0\\
\nonumber&  \epsilon_{3} a a_{13} + c (\lambda_{1} - \lambda_{3}) =0\\
\label{Killing-4-n}& b(\lambda_{3} - \lambda_{2}) =0.
\end{align}
\end{lem}

\begin{lem} Assume that $a_{21} = a_{31} =0$, $a_{23} + a_{32} =0$, $H_{23}=0$, $H_{12} = -  \epsilon_{2} a_{12}$ and $H_{13} = - \epsilon_{3} a_{13}$.
The form $F$ defined by (\ref{F-concrete})  is of type $(1,1)$ with respect to $J_{-}$ if and only if the following relations hold:
\begin{align}
\nonumber&  \epsilon_{2} a a_{12} -  d ( H_{123} -\lambda_{3}) =   \epsilon_{2}  ( ca_{23} - d a_{33}  +  b a_{13})\\
\label{F-11-concrete}&  \epsilon_{3} a a_{13} + c(H_{123} -\lambda_{2} ) = \epsilon_{2}  ( ca_{22}   + da_{23} - ba_{12}).
\end{align}
\end{lem}

\begin{proof} 
From the definition of $J_{-}$, a  $2$-form as in (\ref{H-F-4-uni}) is of type $(1,1)$ with respect to $J_{-}$ if and 
only of 
\begin{equation}\label{f-12}
F_{12} = -  F_{3},\ F_{13} =  F_{2}.
\end{equation}
Using relation  (\ref{F-concrete}),  a simple computation shows that
\begin{align}
\nonumber& F_{12} = \frac{1}{2c_{+}} ( - \epsilon_{2} a a_{12} + d (H_{123} - \lambda_{3} ) )\\
\nonumber& F_{13} = - \frac{1}{2c_{+}} ( \epsilon_{3}  a a_{13} + c  (H_{123} - \lambda_{2} ) )\\
\nonumber& F_{1} =-\frac{1}{2c_{+}} (  \epsilon_{1} b a_{11} + 2 \epsilon_{2} c a_{12} + 2 \epsilon_{2} d a_{13})\\
\nonumber& F_{2} =-\frac{\epsilon_{2} }{2c_{+}} (  c  a_{22} +  d a_{23}  - b a_{12})\\
\label{Expr-F}& F_{3} =\frac{\epsilon_{2} }{2c_{+}} ( c a_{23} -  d  a_{33} +   b a_{13}).
\end{align}
The claim follows from (\ref{Expr-F}) and (\ref{f-12}).
\end{proof}

\begin{rem}{\rm For later use, remark that
\begin{equation}\label{F-23}
F_{23} = \frac{1}{2 c_{+}} b ( H_{123} - \lambda_{1}).
\end{equation}}
\end{rem}

The next lemma  describes the Lie bracket of the Lie algebra $\mathfrak{g}$, 
the Killing field
$X_{+} $,   the forms $H$ and $F$,  
such that all conditions from 
Corollary \ref{ex-4-dim} are satisfied, except the normalization 
$g(X_{+}, X_{+}) =1- c_{+}^{2}$ and the conditions  which involve $X_{-}$.

\begin{lem}\label{classes-initial}  Let $\epsilon_1, \epsilon_2=\epsilon_3\in \{\pm 1\}$ and $c_+ \in \mathbb{R}\setminus \{ 0\}$. There are eight  classes 
$(a_{ij}, \lambda_{i}, H_{123}$, $H_{ij}, F_{ij}, F_{i}, a, b, c, d)$  of solutions
of the systems (\ref{cond-deriv-a}), (\ref{H-cl}), 
(\ref{F-cl}), 
(\ref{cond-coeff}),  (\ref{Killing-4-n}),  (\ref{F-11-concrete}), (\ref{Expr-F}), (\ref{F-23})
with  at least one of the constants $a$, $b$, $c$,  $d$ non-zero and excluding the case 
$\lambda_1=\lambda_2=\lambda_3=0$, i.e.\ 
the case when the ideal $\mathfrak{g}_0\subset\mathfrak{g}$ is abelian.
They are given as follows
(below the constants  $(H_{123}, H_{ij}, F_{ij}, F_{i}, a$, $b, c, d)$ 
are incorporated in the corresponding Killing field $X_{+}$,  and forms $H$ and $F$):\

1) $a_{12} = a_{13} =a_{ii }=0$ for any $i$,   $a_{23}\in \mathbb{R}$, 
$\lambda_{1} \in \mathbb{R}\setminus \{ 0\}$, $ \lambda_{2} =\lambda_{3}\in \mathbb{R}\setminus \{  \lambda_{1}\}$, 
$X_{+} = au + be_{1}$  (with  $a\in \mathbb{R}\setminus \{ 0\}$ and $b\in \mathbb{R}$), 
and
\begin{equation}\label{case-1}
H = -\lambda_{1} e_{1}^{*}\wedge e^{*}_{2}\wedge e^{*}_{3},\ F = - \frac{b\lambda_{1}}{ c_{+} } e^{*}_{2}\wedge e^{*}_{3}.
\end{equation}

2) $a_{12} = a_{13} =a_{ii }=0$ for any $i$,   $a_{23}\in \mathbb{R}$, 
$\lambda_{1}= \lambda_{2} =\lambda_{3}  \in \mathbb{R}\setminus \{ 0\}$, 
$X_{+} = au + be_{1}$ 
(with $a \in \mathbb{R}\setminus \{ 0\}$ 
and $b\in \mathbb{R}$). The
forms $H$ and $F$ are given by (\ref{case-1}).\

3) $a_{12} = -\frac{d\lambda_{2} \epsilon_{2}}{a}$, $a_{13} = \frac{ c\lambda_{2}\epsilon_{3}}{a}$,  $a_{ii}=0$ for any $i$, 
$a_{23} \in \mathbb{R}$, 
$\lambda_{1} =0$, $\lambda_{2} =\lambda_{3}\in  \mathbb{R}\setminus \{ 0\}$, $X_{+} = au -\frac{ \epsilon_{3} a a_{23}}{ \lambda_{2}} e_{1} + ce_{2} + de_{3}$ (with $a\in \mathbb{R}\setminus \{ 0\}$ and  $c, d\in \mathbb{R}$) 
and
\begin{equation}
H = \frac{\lambda_{2}}{a} u^{*}\wedge ( d e_{1}^{*} \wedge e_{2}^{*} - c e_{1}^{*} \wedge e_{3}^{*}),\ F =0.
\end{equation}

4) $a_{12} = a_{13} = a_{ii }=0$ for any $i$, $a_{23}\in \mathbb{R}$, 
$\lambda_{1} =0$, $\lambda_{2} = \lambda_{3} \in \mathbb{R}\setminus \{ 0\}$, 
$X_{+ } = a u + b e_{1}$ (with  $a\in \mathbb{R}\setminus \{ 0\} $ 
and $b\in \mathbb{R}\setminus \{ -\frac{\epsilon_{2} a a_{23}}{\lambda_{2}} \}$). The form $H$ and $F$ are trivial.\

5) $a_{12} = a_{13} = a_{ii} =0$ for any $i$,  $a_{23}\in \mathbb{R}$,  $\lambda_{1} =\lambda_{2} =\lambda_{3} \in \mathbb{R}\setminus \{  0\}$, $X_{+} = b e_{1}$ (with $b\in \mathbb{R}\setminus \{ 0\}$) 
and
\begin{equation}
H = -\lambda_{1} e_{1}^{*}\wedge e^{*}_{2} \wedge e^{*}_{3},\ F = -\frac{b\lambda_{1}}{ c_{+}} e^{*}_{2} \wedge e^{*}_{3}.
\end{equation}

6) $a_{23} =a_{33} =0$, $a_{11} = -\epsilon_{2}  \lambda_{3}$, $a_{22} = \epsilon_{2} \lambda_{3}$, 
$a_{12}, a_{13}\in \mathbb{R}$, 
$\lambda_{1} = \lambda_{2} =0$, $\lambda_{3} \in \mathbb{R}\setminus \{ 0\}$, $X_{+} = d e_{3}$  
(with $d\in \mathbb{R}\setminus \{ 0\}$) 
and 
\begin{align}
\nonumber& H = \lambda_{3} e^{*}_{1}\wedge e^{*}_{2} \wedge e^{*} _{3}-  \epsilon_{2} u^{*} \wedge ( a_{12} e^{*}_{1}\wedge e^{*}_{2} + a_{13} e^{*}_{1}\wedge e^{*}_{3})\\
\nonumber&  F= -\frac{ \epsilon_{3} d a_{13}}{c_{+}} u^{*}\wedge e^{*}_{1}.
\end{align}

 7)  $a_{22} = a_{23} =0$, $a_{11} = -\epsilon_{3}  \lambda_{2}$, 
 $a_{33} = \epsilon_{3}  \lambda_{2}$, 
 $a_{12}, a_{13}\in \mathbb{R}$, $\lambda_{1} = \lambda_{3} =0$, $\lambda_{2}\in \mathbb{R}\setminus \{ 0\}$, 
 $X_{+} =c e_{2}$  (with $c\in \mathbb{R}\setminus \{ 0\}$) 
 and
 \begin{align}
 \nonumber& H= \lambda_{2} e_{1}^{*}\wedge e^{*}_{2} \wedge e^{*}_{3} -  \epsilon_{2} u^{*}\wedge ( a_{12} e^{*}_{1} \wedge e_{2}^{*} + a_{13} e^{*}_{1} \wedge e^{*}_{3} )\\
 \nonumber&  F = -\frac{ \epsilon_{2} ca_{12}}{c_{+}} u^{*}\wedge e^{*}_{1}.
 \end{align}

 8) $a_{12} = a_{13} = a_{ii} =0$ for any $i$, 
 $a_{23}\in \mathbb{R}$,  
 $\lambda_{1} \in \mathbb{R}$, 
 $\lambda_{2} =\lambda_{3} \in \mathbb{R} \setminus \{  \lambda_{1}\}$, 
  $X_{+} = b e_{1}$  (with $b\in \mathbb{R}\setminus \{ 0\}$) 
  and
 \begin{equation}\label{H-F-8} 
 H = -\lambda_{1} e^{*}_{1} \wedge e^{*}_{2} \wedge e^{*}_{3},\ F = - \frac{b\lambda_{1}}{c_{+}} e_{2}^{*} \wedge e^{*}_{3}.
 \end{equation}

In all of the above cases, $a_{21} = a_{31} =0$ and $a_{32} = - a_{23}$. 
\end{lem}

\begin{proof} 
The claim follows from elementary algebraic computations. 
When  $a\neq 0$, the second
relation (\ref{Killing-4-n}) implies that 
$a_{ii} =0$ for any $i$ and  the third  and last relation (\ref{cond-coeff}) imply that
$\lambda_{2} = \lambda_{3}$ and $H_{123} =- \lambda_{1}.$   The case $a\neq 0$ leads to the first four  classes of the statement.
When $a =0$,  relations (\ref{Killing-4-n}) imply
\begin{equation}
d( \lambda_{2} - \lambda_{1}) =c(\lambda_{1} -\lambda_{3}) =b( \lambda_{2} - \lambda_{3} ) =0.
\end{equation}
In particular,  at least two from the 
$\lambda_{i}$'s coincide (otherwise $X_{+} =0$).  
Considering all  possibilities ($\lambda_{1}=\lambda_{2} = \lambda_{3}$,  $\lambda_{1} = \lambda_{2} \neq \lambda_{3}$, 
$\lambda_{1} = \lambda_{3} \neq\lambda_{2}$ and $\lambda_{2} = \lambda_{3} \neq \lambda_{1}$)
we obtain  the remaining four cases from the statement. 
The expressions of $H$ and $F$ follow from
(\ref{H-F-4-uni}), (\ref{cond-coeff}),  (\ref{Expr-F}) and  (\ref{F-23}).
\end{proof}

The next lemma concludes the proof of Proposition~\ref{4c-n0}.

\begin{lem} Consider the classes of solutions  from Lemma \ref{classes-initial} with  $X_{+}$ such that $0 \neq  g(X_{+}, X_{+}) < 1$
and  $c_{+} \in \mathbb{R}\setminus \{ -1, 0, 1\}$  such that 
$c^{2}_{+} = 1 - g(X_{+}, X_{+})$.
There is a  left invariant vector field $X_{-}$  on $G$ and a left invariant endomorphism $J_{+}\in \mathrm{End}\, (TG)$, such that $( g, 
J_{+}, J_{-}, X_{+}, X_{-}, c_{+})$
are the components of an adapted  $B_{4}$-generalized 
pseudo-K\"{a}hler structure, only in cases 3) (if $c=d=0$), 4) and 8) of Lemma \ref{classes-initial}.  The resulting $B_{4}$-generalized pseudo-K\"{a}hler structures from cases  3), 4) and 8)  (with $a_{23}$ replaced by $\beta$)
are described in Proposition~\ref{4c-n0}.
\end{lem}

\begin{proof}  
We need to determine  $X_{-}\in \mathfrak{g}$ and  $J_{+}\in \mathrm{End}\, ( \mathfrak{g}) $ 
such that $X_{-}$ is  orthogonal 
to $X_{+}$ and has the same norm as $X_{+}$, 
$J_{+}$ is $g$-skew-symmetric, $J_{+} X_{+} = - c_{+} X_{-}$, $J_{+} X_{-} = c_{+} X_{+}$, 
$J_{+}\vert_{ \{ X_{+}, X_{-} \}^{\perp} }$ is a complex structure 
and
\begin{align}
\nonumber& \nabla_{X} X_{-} +\frac{1}{2} H(X_{-}, X) = - J_{+} F(X),\ \forall X\in \mathfrak{g},\\
\label{left} &  H(X_{+}, X_{-}) = c_{+} F(X_{-})
\end{align}
(see Corollary  \ref{ex-4-dim}). 
Consider
case 1) from Lemma \ref{classes-initial}. Then $X_{+} = au + b e_{1}$ with $a\neq 0$ and 
\begin{equation}\label{F-case-1}
F(u) = F(e_{1}) =0,\ F(e_{2}) = -\frac{ \epsilon_{3} \lambda_{1} b}{ c_{+}} e_{3},\ F(e_{3} )=\frac{ \epsilon_{2} \lambda_{1} b }{c_{+}} e_{2}.
\end{equation}
Let  $X_{-}= \tilde{a} u + \tilde{b} e_{1} +\tilde{c}e_{2} +\tilde{d}e_{3}$, where $\tilde{a}, \tilde{b}, \tilde{c}, \tilde{d}\in \mathbb{R}$.
We compute  
\begin{align}
\nonumber& \nabla_{e_{1}} X_{-} +\frac{1}{2} H(X_{-}, e_{1}) = \epsilon_{3} \lambda_{2} ( - \tilde{d} e_{2} + \tilde{c} e_{3})\\
\nonumber& \nabla_{e_{2}} X_{-} +\frac{1}{2} H(X_{-}, e_{2}) =  \lambda_{1} ( \epsilon_{1} \tilde{d} e_{1} -  \epsilon_{3} \tilde{b} e_{3})\\
\nonumber& \nabla_{e_{3}} X_{-} +\frac{1}{2} H(X_{-}, e_{3}) =  \lambda_{1} ( - \epsilon_{1} \tilde{c} e_{1} +\epsilon_{2}  \tilde{b} e_{2})\\
\label{D-minus} & \nabla_{u} X_{-} +\frac{1}{2} H(X_{-}, u) = a_{23}  ( - \tilde{d} e_{2} + \tilde{c} e_{3}).
\end{align}
Combining the first relation (\ref{left}) with (\ref{F-case-1}) and (\ref{D-minus}), we obtain 
\begin{align}
\nonumber& a_{23} ( -\tilde{d} e_{2} + \tilde{c} e_{3} ) =0,\ \lambda_{2} (-\tilde{d} e_{2} +\tilde{c} e_{3} ) =0\\
\nonumber& \frac{\epsilon_{3} b}{ c_{+}} J_{+} e_{3} = \epsilon_{1} \tilde{d} e_{1}- \epsilon_{3} \tilde{b} e_{3},\ \frac{\epsilon_{2} b}{ c_{+}} J_{+} e_{2} =  \epsilon_{1}  \tilde{c} e_{1} - \epsilon_{2} \tilde{b} e_{2} .
\end{align}
Since $J_{+}$ is $g$-skew and the basis $\{ u, e_{1}, e_{2}, e_{3} \}$ is orthonormal, $\tilde{b}=0$ and the second line above becomes 
\begin{equation}\label{C-ref}
b J_{+} e_{3} = \epsilon_{1} \epsilon_{3} c_{+} \tilde{d} e_{1},\ b J_{+} e_{2} = \epsilon_{1} \epsilon_{2} c_{+} \tilde{c} e_{1}.
\end{equation}
On the other hand $\tilde{b}=0$ and  $g (X_{+}, X_{-}) =0$ imply  that $\tilde{a}=0$.  Thus   $X_{-}= \tilde{c} e_{2} + \tilde{d} e_{3}.$ 
On the other hand, relation (\ref{C-ref}) also implies that $b\neq 0$ (otherwise $X_{-} =0$). 
It is now straightforward to check that relations (\ref{C-ref}) combined  with $J_{+} (X_{-}) = c_{+} X_{+}$  and $c_{+} a\neq 0$ 
lead to a contradiction. 
The remaining cases from Lemma \ref{classes-initial} can be treated similarly. 
\end{proof}

While this section was concerned with $B_{4}$-generalized K\"{a}hler  structures with $c_{+} \neq 0$,  it remains  to find examples of such structures
with $c_{+}=0$. 
Several hints in this direction are given below.

\begin{exa}{\rm i) By   rescaling  (see Corollary \ref{rescaling-2}),   we obtain from  Proposition~\ref{4c-n0} adapted $B_{4}$-generalized pseudo-K\"{a}hler structures with corresponding vector fields $X_{\pm}$ of norm one 
(i.e.\ $c_{+}=0$).\

ii) Using similar computations as above, one can  construct  further left-invariant $B_{4}$-generalized pseudo-K\"{a}hler structures
with $X_{\pm}$ of norm one.
The $2$-form $F$ satisfies  relations (\ref{H-cl}), (\ref{F-cl}) and (\ref{f-12}) also in this case,  but it is no longer  related
to $X_{+}$  by relation (\ref{F-concrete}).  
}
\end{exa}

V. Cort\'es: vicente.cortes@math.uni-hamburg.de\

Department of Mathematics and Center for Mathematical Physics, University of Hamburg,  Bundesstrasse 55, D-20146, Hamburg, Germany.\\

L. David: liana.david@imar.ro\

Institute of Mathematics  'Simion Stoilow' of the Romanian Academy,   Calea Grivitei no. 21,  Sector 1, 010702, Bucharest, Romania.

\end{document}